\documentclass[a4paper,11pt]{amsart}

\usepackage{amsmath}
\usepackage{amsfonts}
\usepackage{amssymb}
\usepackage{graphicx}
\usepackage[abbrev,alphabetic]{amsrefs}
\usepackage{color}
\usepackage{soul,xcolor} 
\setstcolor{red}
\usepackage{mathtools}

\usepackage{amsthm}
\usepackage{comment}
\usepackage[all,cmtip]{xy}
\usepackage{tikz-cd}
\usetikzlibrary{cd}

\makeatletter
\newcommand*{\relrelbarsep}{.386ex}
\newcommand*{\relrelbar}{%
  \mathrel{%
    \mathpalette\@relrelbar\relrelbarsep
  }%
}
\newcommand*{\@relrelbar}[2]{%
  \raise#2\hbox to 0pt{$\m@th#1\relbar$\hss}%
  \lower#2\hbox{$\m@th#1\relbar$}%
}
\providecommand*{\rightrightarrowsfill@}{%
  \arrowfill@\relrelbar\relrelbar\rightrightarrows
}
\providecommand*{\leftleftarrowsfill@}{%
  \arrowfill@\leftleftarrows\relrelbar\relrelbar
}
\providecommand*{\xrightrightarrows}[2][]{%
  \ext@arrow 0359\rightrightarrowsfill@{#1}{#2}%
}
\providecommand*{\xleftleftarrows}[2][]{%
  \ext@arrow 3095\leftleftarrowsfill@{#1}{#2}%
}
\makeatother

\newcommand{\stacksproj}[1]{{\cite[Tag~{#1}]{stacks-project}}}

\newcommand{\Zpp}{\mathbb{Z}_{(p)}}

\newcommand{\mbA}{\mathbb{A}}
\newcommand{\mbN}{\mathbb{N}}
\newcommand{\mbQ}{\mathbb{Q}}
\newcommand{\mbF}{\mathbb{F}}

\newcommand{\mbC}{\mathbb{C}}

\newcommand{\mbZ}{\mathbb{Z}}

\newcommand{\mcE}{\mathcal{E}}
\newcommand{\mcF}{\mathcal{F}}

\newcommand{\mcI}{\mathcal{I}}
\newcommand{\mcK}{\mathcal{K}}
\newcommand{\mcL}{\mathcal{L}}
\newcommand{\mcN}{\mathcal{N}}
\newcommand{\mcM}{\mathcal{M}}
\newcommand{\mcO}{\mathcal{O}}

\newcommand{\red}{\mathrm{red}}
\newcommand{\perf}{\mathrm{perf}}

\DeclareMathOperator{\Mor}{Mor}
\DeclareMathOperator{\Mob}{Mob}
\DeclareMathOperator{\Supp}{Supp}
\DeclareMathOperator{\Spec}{Spec}

\DeclareMathOperator{\Hom}{Hom}

\DeclareMathOperator{\Pic}{Pic}

\DeclareMathOperator{\ex}{\mathbb{E}}
\newcommand*{\coloneq}{\mathrel{\mathop:}=}
\newcommand*{\PicS}{\mathcal{P}\! \mathit{ic}}

\theoremstyle{plain}
\newtheorem{theorem}{Theorem}[section]
\newtheorem{workinprogress}[theorem]{Work in progress}
\newtheorem{proposition}[theorem]{Proposition}
\newtheorem{lemma}[theorem]{Lemma}
\newtheorem{corollary}[theorem]{Corollary}
\newtheorem{conj}[theorem]{Conjecture}

\newtheorem{claim}[theorem]{Claim}

\theoremstyle{definition}
\newtheorem{definition}[theorem]{Definition}
\newtheorem{example}[theorem]{Example}

\theoremstyle{remark}
\newtheorem{remark}[theorem]{Remark}

\title[Keel's theorem and quotients in mixed characteristic]{Keel's base point free theorem and quotients in mixed characteristic}

\author{Jakub Witaszek} 
\address{Department of Mathematics \\  
University of Michigan\\  
Ann Arbor, MI 48109, USA}
\email{jakubw@umich.edu}

\begin{document}

\begin{abstract}
We develop techniques of mimicking the Frobenius action in the study of universal homeomorphisms in mixed characteristic. As a consequence, we show a mixed characteristic Keel's base point free theorem obtaining applications towards the mixed characteristic Minimal Model Program, we generalise Koll\'ar's theorem on the existence of quotients by finite equivalence relations to mixed characteristic, and we provide a new proof of the existence of quotients by affine group schemes.
\end{abstract}


\maketitle

\section{Introduction}

There are three natural classes of algebraic varieties: of characteristic zero, of positive characteristic, and of mixed characteristic. In trying to understand characteristic zero varieties one can apply a wide range of techniques coming from analytic methods like  vanishing theorems. More complicated though they are, positive characteristic varieties come naturally with the Frobenius action which often allows for imitating analytic proofs or sometimes even showing results which are false over $\mbC$. Of all the three classes, the mixed characteristic varieties are the most difficult to understand as they represent the worst of both worlds: one lacks the analytic methods and the Frobenius action when working with them. Recent years have seen a surge of interest in the study of geometry and commutative algebra of mixed characteristic varieties (cf.\ \cites{andre18,bhatt16,ms18,maschwede18,tanaka16_excellent,EH16}) as they bridge the gap between positive and zero characteristics and play a central role in  number theory. 

What allows for many of the applications of Frobenius is the following observation: if $f \colon X \to Y$ is a universal homeomorphism  of positive characteristic schemes (for example, a thickening), then its perfection $f_{\perf} \colon X_{\perf} \to Y_{\perf}$ is an isomorphism. The goal of this article is to introduce analogues of this fact in mixed characteristic and employ them to generalise many positive characteristic results, with focus on two main sources of applications: the study of base point freeness and constructing quotients.

Before moving on to mixed characteristics, let us give one prominent example of the efficacy of the Frobenius action in positive characteristic: Keel's base point free theorem.
\begin{theorem}[\cite{keel99}] Let $L$ be a nef line bundle on a projective scheme $X$  defined over a positive 
characteristic field $k$. Let $\mathbb{E}(L)$ be the union of all integral subschemes on which $L$ is not big. Then $L$ is semiample if and only if $L|_{\mathbb{E}(L)}$ is so.
\end{theorem}
Here, a line bundle $L$ is nef if $L \cdot C \geq 0$ for every proper curve $C \subseteq X$, it is semiample if some multiple of it is base point free, and it is big {(in the case of the scheme being integral)} if $L^{\otimes m} \otimes A^{-1}$ admits a section for some ample line bundle $A$ and some $m \in \mbN$. 

This seminal result plays a vital role in the study of positive characteristic geometry as it allows for showing base point freeness by an inductive argument.  It is indispensable in the positive characteristic Minimal Model Program (\cite{hx13}), but has many other applications: to birational geometry (e.g.\ \cite{CMM14,birkar17,ct17,MNW15}), moduli spaces of curves (e.g.\ \cite{keel99,keel03}), arithmetic moduli (\cite{BS17}), or Mumford's conjecture (\cite{SP11}) to mention a few. Surprisingly, this result is false in characteristic zero.




In this article, we generalise Keel's theorem to mixed characteristics. In particular, this provides a positive answer to a problem posed by Seshadri (\cite[Remark 2]{seshadri05}).
\begin{theorem}[Theorem \ref{theorem:main}] \label{theorem:main_intro}  Let $L$ be a nef line bundle on a scheme $X$ projective over an excellent base scheme $S$. Then $L$ is semiample over $S$ if and only if both $L|_{\mathbb{E}(L)}$ and $L|_{X_{\mbQ}}$ are so.
\end{theorem}
Here $X_{\mbQ} \coloneq X \times_{\Spec \mbZ} \Spec \mbQ$. Note that the assumption that $L|_{X_{\mbQ}}$ is semiample is necessary, because Keel's result by itself is false in characteristic zero. Further, we prove an analogous result for semiampleness replaced by EWM (\emph{endowed-with-a-map}, see Section \ref{s:preliminaries}).

As a corollary of Theorem \ref{theorem:main_intro}, we show that contractions exist in the mixed characteristic Minimal Model Program (see Corollary \ref{cor:bpf_plt}) and prove the following base point free theorem.
\begin{corollary}[{Corollary \ref{cor:bpf}}] \label{cor:bpf_intro} Let $S$ be a spectrum of a mixed characteristic Dedekind domain with residue fields of closed points {being locally finite}. Let $(X,\Delta)$ be a klt pair on a normal integral scheme $X$ of (absolute) dimension three which is projective {and surjective} over $S$ and let $L$ be a nef {and big} Cartier divisor on $X$ such that $L-(K_X+\Delta)$ is nef and big. Then $L$ is semiample.
\end{corollary}
\noindent {A field is \emph{locally finite} if it is a subfield of $\overline{\mathbb{F}}_p$ for some prime number $p>0$. When the residue fields of $S$ are not locally finite, we prove that $L$ is EWM.}

We move on to constructing quotients of schemes. The following result in positive characteristic has been shown by Koll\'ar (see \cite{kollar12}).
\begin{theorem} \label{thm:quotients_intro} Let $X$ be a separated algebraic space of finite type over an excellent base scheme $S$. Let $\sigma \colon E \rightrightarrows X$ be a finite, set theoretical equivalence relation and assume that the geometric quotient $X_{\mbQ} / E_{\mbQ}$ exists as a separated algebraic space of finite type over $S$. Then the geometric quotient $X / E$ exists as a separated algebraic space of finite type over $S$.
\end{theorem}
Note that as with Keel's theorem, the quotients by set theoretical finite equivalence relations need not exist in characteristic zero. However, one can construct them in many important cases (see \cite{kollar12,kollar13}).

Lastly, we provide a new proof of the following result (cf.\ \cite[Conjecture 1.1]{kollar97}).
\begin{theorem}[{\cite[Theorem 1.1 and Corollary 1.2]{km97}}] \label{thm:quotients_algebraic_groups_intro}Let $G$ be an affine  algebraic group scheme of finite type and smooth over an excellent base scheme $S$ and let $X$ be a separated algebraic space of finite type over $S$. Further, let $m \colon G \times X \to X$ be a proper $G$-action on $X$. Then a geometric quotient $X/G$ exists and is a separated algebraic space of finite type over $S$.
\end{theorem} 
The assumption on the smoothness of $G$ can be weakened (cf.\ Remark \ref{rem:flat_group_quotients}). Over $\mbC$ the above fundamental theorem was proved in \cite{popp73}. Building on the results of Seshadri (\cite{seshadri72}), Koll\'ar showed this theorem for algebraic spaces over positive characteristic fields, and also for mixed characteristic ones when the group scheme is reductive (\cite{kollar97}). Finally, the conjecture has been settled in \cite{km97}, where it was shown that quotients by flat groupoids with finite stabilisers exist. Although the above result is known to hold for the last two decades, we believe it is interesting to provide a new proof of it, one which follows Koll\'ar's original strategy.


We finish this part of the introduction by explaining an important recurring theme in the proofs of all the above results:  constructing pushouts of diagrams $X \xleftarrow{p} Y \xrightarrow{g} Y'$ where $g$ is a universal homeomorphism. 
In order to prove his remarkable result (\cite{kollar97}), Koll\'ar showed that such pushouts exist in positive characteristic, and in mixed characteristic as well if $g$ is, in addition, an isomorphism over $\mbQ$ and $p$ is finite. The following generalisation of his result plays a vital role in the proofs of the above theorems, and we believe is interesting in itself.

\begin{theorem}[{cf.\ Theorem \ref{thm:general_pushouts_in_mixed_char}}] \label{thm:general_pushouts_in_mixed_char_intro} Let $X \xleftarrow{p} Y \xrightarrow{g} Y'$ be a diagram of schemes or algebraic spaces such that $p$ is representable, quasi-compact, and \emph{separated}, and $g$ is a representable universal homeomorphism. Assume that a pushout of $X_{\mbQ} \leftarrow Y_{\mbQ} \rightarrow Y'_{\mbQ}$ exists. Then so does a pushout of $X \leftarrow Y \to Y'$.
\end{theorem}
An important case of this theorem is constructing pushouts of $X \xleftarrow{p} X_{\mbQ} \xrightarrow{g} X'_{\mbQ}$, where $g$ is a representable universal homeomorphism, in other words, extending a universal homeomorphism from characteristic zero to mixed characteristic. 

Having constructed such pushouts, we need to study their properties and the following theorem allows for finding line bundles on them. Here, $\PicS_X$ denotes the groupoid of line bundles on $X$.  
\begin{theorem}  \label{thm:univ_homeo_and_Pic} \label{thm:univ_homeo_and_Pic_intro}
Let $f \colon X \to Y$ be a finite universal homeomorphism of Noetherian schemes over $\Zpp$. Then the following diagram 
\begin{center}
\begin{tikzcd}
\PicS_Y[\frac{1}{p}] \arrow{r}{f^*} \arrow{d} & \PicS_X[\frac{1}{p}] \arrow{d} \\
\PicS_{Y_{\mbQ}}[\frac{1}{p}] \arrow{r} & \PicS_{X_{\mbQ}}[\frac{1}{p}],
\end{tikzcd}
\end{center}
is Cartesian in the $2$-category of groupoids. 
\end{theorem}

\subsection{Further discussion} \label{ss:fdiscussion}
In this subsection we summarise other topics related to our study of Keel's theorem and quotients in mixed characteristic. In order to prevent this paper from becoming too long, we decided not to pursue them in detail here. Instead, we hope to address some of  them in forthcoming articles (e.g.\ \cite{witaszekinprogress}).

\subsubsection*{Mumford conjecture (Haboush's theorem)}

Given an affine scheme $\Spec A$ which is finitely generated over a characteristic zero field, and a reductive group $G$ acting on $\Spec A$, it is easy to show using the averaging operator that $A^G$ is finitely generated as well. However, in general this has been an open problem for many years (known as Mumford's conjecture), eventually settled in positive characteristic by Haboush using Steinberg's representations (\cite{haboush75}), and extended to mixed characteristic using similar methods by Seshadri (\cite{seshadri77}). Before Haboush's seminal paper, Seshadri set up a program for showing Mumford's conjecture by geometric means (cf.\ \cite{seshadri72}). After the announcement of \cite{keel99}, Seshadri realised (\cite{seshadri05}) that Keel's base point free theorem is exactly what is needed to conclude his program in positive characteristic and yield a geometric proof of Mumford's conjecture in this setting (this was eventually proven together with Sastry in \cite{SP11}). As remarked by Seshadri, the missing component for concluding his program in full generality is a mixed characteristic variant of Keel's theorem. Hence, the results of our paper should possibly allow for a geometric proof of Mumford's conjecture in this general setting.

\subsubsection*{Mixed characteristic Minimal Model Program}

 We were motivated to seek a mixed characteristic variant of Keel's theorem by our study of the higher dimensional mixed characteristic Minimal Model Program (see \cite{tanaka16_excellent} for the two-dimensional case). The recent breakthrough proof of the validity of the Minimal Model Program for positive characteristic threefolds by Hacon and Xu (see \cite{hx13}) is based on two main components: Keel's theorem used to construct contractions, and the Frobenius regularity used to construct flips. Our mixed characteristic Keel's theorem provides exactly what is needed to generalise the former component to mixed characteristic. Note that, recently, Schwede and Ma (\cite{ms18}), motivated by the work of Andr\'e, Bhatt, and Scholze (\cites{andre18,bhatt16,Scholze}), introduced a mixed characteristic analogue of F-regularity. We hope that this could be used to mimick Hacon-and-Xu's proof of the existence of flips (see \cite{MSTWW19} for the first step in this direction) which combined with our results would yield the validity of the MMP for mixed characteristic threefolds.


\subsubsection*{Relative semiampleness} Cascini and Tanaka have shown that given a projective morphism $f \colon X \to Y$ of positive characteristic Noetherian schemes, the relative semi-ampleness of a line bundle $L$ on $X$ may be verified fibrewise (\cite{ct17}, see \cite[Theorem 1.3]{BS17} for a similar result). The three main components of their proof are: Keel's base point free theorem, Kollar's existence of quotients by finite equivalence relations, and ``gluing'' of semiampleness of line bundles. In this article we generalise these components to mixed characteristic (Theorem \ref{theorem:main_intro}, Theorem \ref{thm:quotients_intro}, and Subsection \ref{ss:gluing}) and, as far as we understand, this will be enough to extend the result of Cascini and Tanaka to morphisms over $\mbZ$. As a corollary, one gets the following.
\begin{workinprogress} \label{thm:relative_semiampleness_mixed} Let $L$ be a nef line bundle on a scheme $X$ projective over a Noetherian base scheme $S$. Assume that $L|_{X_{\mbF_p}}$ and $L|_{X_{\mbQ}}$ are semiample over $S$. Then $L$ is semiample over $S$.
\end{workinprogress}
\noindent Due to the amount of technical details, we do not sort out the proof here, but instead postpone it to a separate article (\cite{witaszekinprogress}). 

\subsubsection*{Moduli spaces of curves}
One of the consequences of Keel's seminal paper was the proof that the relative canonical divisor on the universal family of curves over $\overline{M}_{g,n}$ is always semiample in positive characteristic (\cite[Theorem 0.4]{keel99}). In \cite{keel03} it was shown that many other nef line bundles on $\overline{M}_{g,n}$ in positive characteristic are semiample and the results of our paper should allow for proving that some of these line bundles (for example corresponding to $K_X$-negative extremal rays, cf.\ \cite[Section 7]{gibney09}) are semiample in mixed characteristics as well. Theorem \ref{thm:relative_semiampleness_mixed} will reduce this problem to the independent study of the characteristic zero and the positive characteristic cases, thus we postpone writing any proofs to \cite{witaszekinprogress}. 

In general, Keel conjectured that every nef line bundle on $\overline{M}_{g,n}$ in positive characteristic is semiample. If this is true, then Theorem \ref{thm:relative_semiampleness_mixed} will imply the following. 
\begin{conj}\label{conj:mgn} Let $\overline{M}_{g,n}$ be the moduli space of genus $g$ curves with $n$ marked points over $\Spec \mbZ$, and let $L$ be a nef line bundle on it. If $L$ is semiample on $\overline{M}_{g,n} \times \Spec \mbQ$, then it is semiample.
\end{conj}

\subsubsection*{K-theory} It is natural to enquire if the constructions of this paper can be extended to a more general framework. In fact, motivated by some of the ideas contained here, one can provide results towards the study of mixed characteristic K-theory (\cite{aemw}). This, as well as Theorem \ref{thm:univ_homeo_and_Pic_intro} and Lemma \ref{lemma:main}, suggest to explore ind-objects of the form $\varinjlim_{\tilde F} W$, where $W$ is a (derived) $\mathbb{Z}_{p}$-stack endowed with a lift $\tilde F$ of Frobenius. Such objects are in some sense complementary to perfectoid spaces, which often come from a similar construction but with the direct limit replaced by the inverse limit. We hope this and other related problems, such as the behaviour of derived Brauer stacks under universal homeomorphisms in mixed charactersitic, to be addressed in a separate article.

\subsection{The idea of the proof of Theorem \ref{theorem:main_intro}}
The key components in the proofs of the main results are Theorem \ref{thm:general_pushouts_in_mixed_char_intro}, Theorem \ref{thm:univ_homeo_and_Pic_intro}, and the ``mixed characteristic multiplicative perfection''. In what follows we explain the last concept by giving a sketch of the proof of Theorem \ref{theorem:main_intro}. Note that \cite{BS19} defined a mixed characteristic perfection in the category of derived schemes; however, their theory seems useful for a different type of geometric applications (see e.g.\ \cite[Appendix]{MSTWW19}).

Using Keel's strategy and Theorem \ref{thm:univ_homeo_and_Pic}, we can deduce Theorem \ref{theorem:main_intro} from the following result. 
\begin{theorem} \label{thm:reduced_semiample} Let $L$ be a nef line bundle on a scheme $X$ projective over a Noetherian base scheme $S$. Then $L$ is semiample (or EWM) if and only if both $L|_{X^{\red}}$ and $L|_{X_{\mbQ}}$ are so, where $X^{\red}$ is the reduction of $X$.
\end{theorem}
In fact, Birkar showed that {for varieties over a field} there exists a thickening $\mathbb{E}(L)^{\rm th}$ of $\mathbb{E}(L)$ such that $L$ is semiample if and only if $L|_{\mathbb{E}(L)^{\rm th}}$ is so (\cite[{Theorem 1.5}]{birkar17}). However, the main difficulty with applying this result in practice is that it is usually difficult to verify that a line bundle on a non-reduced scheme is semiample.

By localising at primes $p \in \Spec \mbZ$, we can assume that $X$ is defined over $\Zpp$. Let us explain the proof of Theorem \ref{thm:reduced_semiample} under the assumption that $X_{\mbQ} = \emptyset$ that is $X$ is defined over $\mbZ/p^m\mbZ$ for some $m>0$. Therewith, we claim that $\mcO_X \to \mcO_{X^{\red}}$ is an isomorphism up to raising the sections to the $p^n$-th power for some $n\gg 0$. In particular, when $k$ is divisible by high enough power of $p$ the same holds for 
\[
H^0(X, {L^k}) \to H^0(X^{\rm red}, {L^k|_{X^{\rm red}}}),
\]
and so $L$ is semiample if and only if $L|_{X^{\red}}$ is so (the idea is that we lift sections locally but then these different local lifts glue up to $p^n$-th power by the claim).

To prove the claim, we can work affine locally. Let $\pi \colon R \to R/I$ be a morphism of rings such that $I$ is a locally nilpotent ideal and $R[\frac{1}{p}]=0$, that is $p^m=0$ for some $m>0$. Since $\pi$ is clearly surjective, it is enough to check that it is injective up to raising the sections to some $p^n$-th power, that is:
\begin{center} \emph{for every $r_1, r_2 \in R$ such that $\pi(r_1)=\pi(r_2)$ we have $r_1^{p^n} = r_2^{p^n}$ for some $n>0$ depending on $r_1$ and $r_2$.}
\end{center}
The first condition stipulates that $r_2 = r_1 + t$ for some $t \in I$. Since $I$ is locally nilpotent, $t^{r+1}=0$ for some $r>0$, and so
\[
r_2^{p^n} = r_1^{p^n} + \sum_{i=1}^r {p^n \choose i} r_1^{p^n-i}t^i.
\]
By taking $n \gg 0$ we can assume that $p^m \mid {p^n \choose i}$ for $i\leq r$, and so $r_2^{p^n} = r_1^{p^n}$, concluding the proof of the claim and the theorem when $X_{\mbQ} = \emptyset$. \\

We can formalise the concept of the validity ``up to some $p^n$-th power'' by introducing a handy notion of a perfection of the sheaf $\mcO_X$; we set 
\[
\mcO^{\perf}_X \coloneq \varinjlim_{s \mapsto s^p} \mcO_X.
\]
Since the $p$-th power map is not additive, this object is only a sheaf of multiplicative monoids. 

Given a universal homeomorphism $f \colon X \to Y$ such that $f_{\mbQ}$ is an isomorphism, we show that $\mcO^{\perf}_Y \to f_*\mcO^{\perf}_X$ is an isomorphism (see Lemma \ref{lem:univ_homeo_iso_generic}), from which we infer the following.
\begin{lemma} \label{lemma:main} Let $f \colon X \to Y$ be a universal homeomorphism of schemes  over $\Zpp$. Then the following diagram is Cartesian
\begin{center}
\begin{tikzcd}
\mcO^{\perf}_Y \arrow{r}{f^*} \arrow{d} & \mcO^{\perf}_X \arrow{d} \\
\mcO^{\perf}_{Y_{\mbQ}} \arrow{r} & \mcO^{\perf}_{X_{\mbQ}}
\end{tikzcd}
\end{center}
{in the category of sheaves of monoids.}
\end{lemma}
\noindent This shows that in mixed characteristic the $p$-th power map behaves to some extent as if it was additive. 

Theorem \ref{thm:reduced_semiample} can be proved using Lemma \ref{lemma:main}, Theorem \ref{thm:general_pushouts_in_mixed_char_intro}, and Theorem~\ref{thm:univ_homeo_and_Pic_intro}. 

\section{Preliminaries} \label{s:preliminaries}
We refer to \cite{stacks-project} for basic definitions in scheme theory and to \cite{km98} for basic definitions in birational geometry (see also \cite{kollar13,tanaka16_excellent,ct17}). {We say that $(X,B)$ is a \emph{log pair} if $X$ is a normal excellent scheme of finite dimension admitting a dualising complex, $B$ is an effective $\mbQ$-divisor, and $K_X+B$ is $\mbQ$-Cartier. Given a scheme $X$ we write $X_{\mbQ} \coloneq X \times_{\Spec \mbZ} \Spec \mbQ$ and $X_{\mbF_p} \coloneq X \times_{\Spec \mbZ} \Spec \mbF_p$ for a prime number $p>0$. We say that a connected scheme $X$ is of \emph{mixed characteristic} if $X_{\mbQ} \neq \emptyset$ and $X_{\mbF_p} \neq \emptyset$ for some prime number $p>0$.} Note that every mixed characteristic Dedekind domain (Noetherian normal domain of dimension at most one) is excellent (\stacksproj{07QW}). {Furthermore, schemes of finite type over Noetherian (resp.\ excellent) base schemes are Noetherian (resp.\ excellent), \stacksproj{01T6} (resp.\ \stacksproj{07QU}), and hence quasi-compact and quasi-separated, \stacksproj{01OY,01T7}. If a scheme is excellent, then its normalisation is finite (\stacksproj{0BB5}). Recall that excellent schemes are Noetherian by definition.}

We say that a morphism of schemes $f \colon X \to Y$ is a \emph{contraction} if it is proper, surjective, and $f_* \mcO_X = \mcO_Y$. Let $X$ be a proper scheme over a Noetherian base scheme $S$, let $\pi \colon X \to S$ be the projection, and let $L$ be a line bundle on $X$. If the base scheme is fixed, we drop the prefix ``relatively'' when referring to notions below. We say that $L$ is relatively \emph{nef} if {$\mathrm{deg}(L|_C) \geq 0$} for every proper curve $C \subseteq X$ over $S$, 
it is relatively \emph{base point free} if the natural map $\pi^*\pi_*L \to L$ is surjective, it is relatively \emph{semiample} if some multiple of it is base point free, and it is relatively \emph{big} if {$L|_{X_{\eta}}$ is big for some generic point $\eta \in f(X)$ and the fibre $X_{\eta}$ over $\eta$} {(that is, $h^0(X_{\eta}, L^k|_{X_{\eta}}) > ck^{\dim X_{\eta}}$ for some constant $c$ and all $k$ divisible enough). The notion of bigness is subtle for non-irreducible schemes and so we will essentially use it only when $X$ is integral.}

{We make a few observations. By definition, $L$ is relatively nef if and only if $L|_{X_y}$ is nef for every closed point $y \in Y$ and the fibre $X_y$ over $y$. By \cite[Lemma 2.6]{ct17}, this is equivalent to $L|_{X_y}$ being nef for every point $y \in Y$. Semiampleness of $L$ may be verified Zariski locally (cf.\ \cite[Lemma 2.12]{ct17}). If $\pi$ is projective {and $X$ is integral}, then $L$ is relatively big if and only if  $\pi_*(L^{\otimes m} \otimes A^{-1}) \neq 0$ for a relatively ample line bundle $A$ and some $m>0$. 

\begin{lemma} \label{lem:pullback-of-big} Let $f \colon X \to Y$ be a finite map of integral proper schemes over a Noetherian base scheme $S$. Let $L$ be a relatively nef line bundle on $Y$. Then $L$ is relatively big over $S$ if and only if $f^*L$ is relatively big over $S$.
\end{lemma}
\begin{proof}
By restricting to the generic point of the image of $X$ (and $Y$) in $S$, we may assume that $X$ and $Y$ are defined over a field.
Since $Y$ is irreducible, $L$ is big if and only if $L^{\dim Y}>0$ (cf.\ \cite[Theorem 2.2.16]{lazarsfeld04a}), and analogously for $f^*L$. Since $(f^*L)^{\dim X} = (\deg f)L^{\dim Y}$, the lemma follows.
\end{proof}
}

Further, following \cite{keel99} (cf.\ \cite[p.7]{ct17}), we say that $L$ is relatively \emph{EWM} (\emph{endowed with a map}) if there exists a proper $S$-morphism $f \colon X \to Y$ to an algebraic space $Y$ proper over $S$  such that {an integral closed subscheme $V \subseteq X$ is contracted (that is, $\dim V < \dim f(V)$) if and only if $L|_V$ is not relatively big. In particular, $f$ satifies this condition if and only if its restriction to $X_y$ satisfies this condition for every point $y \in Y$.}


Note that the property of $L$ being EWM can be checked affine locally on $S$. Indeed, given a surjective contraction $f \colon X \to Y$ and a morphism $h \colon X \to Z$ contracting all the geometric fibres of $f$, where $X$, $Y$, $Z$ are proper algebraic spaces over $S$, there exists a unique map $g \colon Y \to Z$ such that $g \circ f = h$ (cf.\ the proof of Proposition \ref{thm:kollar_quotients_exist}). In particular, if $f \colon X \to Y$ is a {contraction} associated to $L$, then $f$ and $Y$ are uniquely determined up to a canonical isomorphism, so any local constructions of maps associated to a line bundle must glue. {The same argument shows that} the Stein factorisation of $f$ is unique.

{For a relatively nef line bundle $L$ on $X$ as above, we define $\mathbb{E}(L)$ to be the union of all closed integral subschemes $V \subseteq X$ such that $L|_V$ is not relatively big over $S$.
\begin{lemma} {Let $X$ be a projective scheme over a Noetherian base scheme $S$.} Then $\mathbb{E}(L)$ is a closed {subset} of $X$.
\end{lemma}
\noindent {In particular, this endows $\mathbb{E}(L)$  with a scheme structure of a reduced subscheme of $X$.}
\begin{proof}
The proof is exactly the same as in \cite[Lemma 2.18]{ct17}. We may assume that $X$ is reduced and, since the problem is local, that $S$ is affine. If $\mathbb{E}(L)=X$, then there is nothing to prove. Thus, we may assume that there exists an irreducible component $X' \subseteq X$ such that $L|_{X'}$ is big. Hence, there exists a relatively ample line bundle $A$ and a line bundle $E$ on $X'$ such that $H^0(X',E) \neq 0$ and $L^{\otimes m}|_{X'} \simeq A \otimes E$. Let $Z$ be a reduced closed subscheme equal to the reduction of the zero set of a section $0 \neq s \in H^0(X',E)$. Then $\mathbb{E}(L) \subseteq Z \cup X''$, where $X = X' \cup X''$ and $X' \not \subseteq X''$. Indeed, if $V \subseteq X'$ is a closed integral subscheme of $X'$ such that $V \not \subseteq Z$, then $L^{\otimes m}|_V \simeq A|_V \otimes E|_V$ is big as $s|_V \in H^0(V, E|_V)$ is non-zero.

In particular, $\mathbb{E}(L) = \mathbb{E}(L|_{Z \cup X''})$ is closed by Noetherian induction.   
\end{proof}}



Let us recall the following pinching result.
\begin{theorem}[{\cite[Theorem 3.1]{Artin70} and \cite[Theorem 38]{kollar12}}] \label{thm:Kollar_pinching} Let $X$ be a Noetherian algebraic space over a Noetherian base scheme $S$, let $Z \subseteq X$ be a closed subspace, and let $g \colon Z \to V$ be a finite surjection {of Noetherian algebraic spaces over $S$.} Then there exists a universal pushout diagram of algebraic spaces 
\begin{center}
\begin{tikzcd}
Z  \arrow[hookrightarrow]{r} \arrow{d}{g} & X \arrow{d}{\pi}  \\
V  \arrow[hookrightarrow]{r} & Y,
\end{tikzcd}
\end{center}  
such that $Y$ is a Noetherian algebraic space over $S$, $V \to Y$ is a closed embedding, and $Z = \pi^{-1}(V)$. Further, the diagram is a pushout square on the level of topological spaces and $\pi$ is a finite map which is an isomorphism over $Y \setminus V$. If $X$, $Z$, and $V$ are of finite type over $S$, then so is $Y$.
\end{theorem}
\begin{proof}
The pushout exists by \cite[Theorem 38]{kollar12}. After an \'etale base change of $Y$, we can assume that the spaces in question are schemes, and so the diagram is a topological pushout and $\pi$ is an isomorphism over $Y \setminus V$ by \stacksproj{0E25}. Further, $Y$ is of finite type, provided that so are $X$,$Z$, and $V$, by \stacksproj{0E27}.
\end{proof}

\subsection{Universal homeomorphisms}
We say that a morphism of schemes {(or algebraic spaces)} $f \colon X \to Y$ is a \emph{universal homeomorphism}, if all of its base changes are homeomorphisms. In the case of schemes, this is equivalent, by \stacksproj{04DF}, to being integral, universally injective, and surjective. However, for algebraic spaces a universal homeomorphism need not be integral (e.g.\ $\mathbb{A}^1/\{x\,{\sim}\, {-}x \,|\, x \neq 0\} \to \mathbb{A}^1/\{x\,{\sim}\, {-}x\} \simeq \mathbb{A}^1$, cf.\ \stacksproj{05Z6}). In this setting, being integral, universally injective, and surjective is equivalent to being a representable universal homeomorphism. 

To verify that a {representable} universally closed (for example, integral) morphism $f$ of schemes {or algebraic spaces} is a universal homeomorphism it is enough to check that $\mathrm{Mor}(\Spec K, X) \to \mathrm{Mor}(\Spec K, Y)$ is a bijection for every algebraically closed field $K$ (cf.\ \stacksproj{01S4 and 03MH}). 

We call an extension of rings $A \subseteq B$ \emph{elementary}, if there exists $b \in B$ such that $A[b] = B$ and $b^2,b^3 \in A$. The following proposition states that universal homeomorphisms in characteristic zero decompose into thickenings and elementary extensions. Indeed, when $f \colon A \to B$ is a universal homeomorphism of {$\mathbb{Q}$-algebras}, then $f$ automatically induces isomorphisms on residue fields {(by base change this reduces to checking that a finite extension of characteristic zero fields $\Spec K \to \Spec L$ is a universal homeomorphism if and only if $K\simeq L$; this follows from the fact that $K \otimes_L K = K^{\oplus \deg L/K}$ as the fields are of characteristic zero)}. 
\begin{proposition}[\stacksproj{0CND}] \label{prop:0CND} \label{prop:elementary_extensions} An extension of rings $A \subseteq B$ is a universal homeomorphism inducing isomorphisms on residue fields if and only if every finite subset $E \subseteq B$ is contained in an extension $A[b_1,\ldots,b_k] \subseteq B$ such that for every $1 \leq i \leq k$ we have that $b_i^2,b_i^3 \in A[b_1,\ldots,b_{i-1}]$.
\end{proposition}
\begin{proposition}[\stacksproj{0CNE}] \label{prop:OCNE} \label{prop:universal_homeomorphisms_affine_extensions}An extension of rings $A \subseteq B$ is a universal homeomorphism if and only if every finite subset $E \subseteq B$ is contained in an extension $A[b_1,\ldots,b_k] \subseteq B$ such that for every $1 \leq i \leq k$ we have that $b_i^2,b_i^3 \in A[b_1,\ldots,b_{i-1}]$ or $pb_i,b_i^p \in A[b_1,\ldots,b_{i-1}]$ for some prime number $p$, {which depends on $i$}.
\end{proposition}
In characteristic $p>0$, universal homeomorphisms may also be  described in the following way. 
\begin{proposition}[{cf.\ \stacksproj{0CNF}, Lemma \ref{lem:univ_homeo_iso_generic}}] \label{prop:univ_homeo_in_char_p} Let $f \colon X \to Y$ be an affine morphism of schemes {of characteristic $p>0$}. Then $f$ is a universal homeomorphism if and only if $f^* \colon \mcO^{\perf}_Y \to \mcO^{\perf}_X$ is an isomorphism.
\end{proposition}
\noindent {Here, $\mcO^{\perf}_X = \varinjlim(\mcO_X \xrightarrow{F} \mcO_X \xrightarrow{F} \ldots)$ denotes the structure sheaf of the perfection $X^{\perf}$ of $X$.}

\begin{remark} \label{rem:scheme_facts} For the convenience of the reader, we recall a few basic scheme theoretic facts that we will often use later on. Here, $S$ is a scheme or an algebraic space (over a scheme) and {$f \colon Y \to X$} is a morphism of schemes or algebraic spaces, respectively.
\begin{enumerate}
	\item Assume that $Y$ is quasi-compact and $X$ is quasi-separated over $S$. Then $f$ is quasi-compact (\stacksproj{03KS}). \label{itm:03KS}
	\item Assume that $Y$ is quasi-compact over $S$ and $f$ is surjective. Then $X$ is quasi-compact over $S$ (cf.\ \stacksproj{03E4}). \label{itm:qc} \label{itm:03E4}

	\item Assume that $Y$ is quasi-separated or separated over $S$. Then so is $f$ (\stacksproj{03KR}). \label{itm:03KR}
	\item \label{itm:09MQ} \label{itm:05Z2} Assume that $f$ is surjective and universally closed. If $Y$ is quasi-separated or separated over $S$, then so is $X$ (\stacksproj{05Z2}).
	\item Assume that $Y$ is locally of finite type over $S$. Then $f$ is locally of finite type (\stacksproj{0462}). \label{itm:0462}

	\item The morphism $f$ is integral if and only if it is affine and universally closed (\stacksproj{01WM, 0415}). \label{itm:01WM} \label{itm:0415}
	\item Assume that $f$ is surjective. If $Y$ is universally closed over $S$, then so is $X$. In particular, if also $X$ is separated and of finite type over $S$, then it is proper (\stacksproj{03GN, 08AJ}). \label{itm:03GN} \label{itm:08AJ}
	\item Assume that $Y$ is proper and $X$ is separated over $S$. Then $f$ is proper (\stacksproj{04NX}). \label{itm:04NX}
	\item Assume that $Y$ is finite (integral, resp.) and $X$ is separated over $S$. Then $f$ is finite (integral, resp.) (\stacksproj{035D}). \label{itm:035D}
	
	\item Assume that $f$ is of finite type with finite fibres and that the algebraic spaces $Y$ and $X$ are quasi-separated over $S$. Then $f$ is quasi-finite (\stacksproj{06RW,0ACK}). \label{itm:06RW}
	\item Assume that $f$ is proper with finite fibres and $X$ is quasi-separated over $S$. Then $f$ is finite (\stacksproj{0A4X}). \label{itm:0A4X}

	\item Assume that $Y$ is affine and $f$ is surjective and integral. Then $X$ is affine (\stacksproj{05YU, 07VT}). \label{itm:05YU} \label{itm:07VT}
	\item Assume that $f$ is integral and induces a bijection $|Y|=|X|$. Then $Y$ is a scheme if and only if $X$ is a scheme (\stacksproj{07VV}). \label{itm:07VV}

	\item Assume that $f$ is a representable universal homeomorphism. Then pulling back induces an equivalence of categories of \'etale or affine \'etale schemes or algebraic spaces over $X$ and $Y$ (\stacksproj{04DZ, 05ZH, 07VW}). \label{itm:04DZ} \label{itm:05ZH} \label{itm:07VW}
\end{enumerate}
\end{remark}

\begin{lemma} \label{lemma:factors_of_universal_homeomorphism} Let $f \colon X \xrightarrow{g} Y \xrightarrow{h} Z$ be morphisms of schemes such that $f$ is a universal homeomorphism. Further, assume that $g$ is surjective, or $g$ is dominant and $h$ is separated. Then both $g$ and $h$ are universal homeomorphisms. 

\end{lemma}
\begin{proof}
First, we show that $g$ is surjective. To this end, we can assume that $g$ is dominant and $h$ is separated. Then, $g$ is integral (Remark \ref{rem:scheme_facts}(\ref{itm:035D})), hence closed and surjective.

Since $g$ is surjective, Remark \ref{rem:scheme_facts}(\ref{itm:03GN}) implies that $h$ is universally closed. Moreover, $h$ is surjective as $f$ is surjective, and it is universally injective as $f$ is universally injective and $g$ is surjective. Therefore, $h$ is a universal homeomorphism. In particular, it is separated, and as above we get that $g$ is integral. Since $f$ is universally injective, $g$ is also universally injective, and so it is a universal homeomorphism.  \qedhere





\end{proof}
\begin{lemma} \label{lem:verify_uh_over_every_p} An affine morphism of schemes $f \colon Y \to X$ is a universal homeomorphism if and only if $f_{\Zpp} \colon Y_{\Zpp} \to X_{\Zpp}$ is a universal homeomorphism for every prime number $p$.
\end{lemma}
Here $X_{\Zpp} \coloneq X \times_{\Spec \mbZ} \Spec \Zpp$.
\begin{proof}
If $f_{\Zpp}$ is a universal homeomorphism for every $p$, then $f$ is universally injective and surjective. To verify integrality, we can assume that $X$ and $Y$ are affine, in which case this follows by \stacksproj{034K}.
\end{proof}

The following lemma allows us to descend finite generatedness under universal homeomorphisms $f \colon Y \to X$. However, when $f^* \colon \mcO_X \to f_*\mcO_Y$ is not injective, the statement is false (cf.\ Remark \ref{remark:pushout_need_not_be_noetherian}).
\begin{lemma}[Eakin-Nagata] \label{lem:fin_gen_when_finite} \label{lem:Eakin-Nagata} Let $f \colon Y \to X$ be an integral morphism of algebraic spaces over a Noetherian base scheme $S$ such that $Y$ is of finite type over $S$ and $f^* \colon \mcO_X \to f_*\mcO_Y$ is injective. Then $X$ is of finite type over $S$. Moreover, if $Y$ is separated or proper over $S$, then so is $X$.
\end{lemma}
\noindent Note that $f$ is automatically finite when $X$ is of finite type.
\begin{proof}
Since $f$ is dominant and closed, it is surjective. Thus, by Remark \ref{rem:scheme_facts}(\ref{itm:03E4}), $X$ is quasi-compact over $S$. To check that $X$ is locally of finite type, we can assume that $X$, $Y$, and $S$ are affine, in which case the statement follows from \cite[Theorem 41]{kollar12}. The separatedness or properness of $X$ provided that of $Y$ is a consequence of Remark \ref{rem:scheme_facts}(\ref{itm:05Z2})(\ref{itm:08AJ}).
\end{proof}

\subsection{Quotients by finite equivalence relations} \label{ss:finite_quotients}

In this subsection we review definitions and basic results on quotients by set theoretic equivalence relations following \cite{kollar12}.

Even in the case of a finite group $G$ acting on a scheme $X$, we cannot expect the quotient $X/G$ to be a scheme unless $X$ is quasi-projective or, more generally, Chevalley-Kleiman (cf.\ \cite[Definition 47]{kollar12}). Therefore, we need to work in the category of algebraic spaces. 

\begin{definition} Let $X$ be a separated algebraic space of finite type over a Noetherian base scheme $S$. A morphism $\sigma \colon E \to X \times_S X$ (equivalently $\sigma_1, \sigma_2 \colon E \rightrightarrows X$ over $S$) is a \emph{set theoretic equivalence relation} on $X$ over $S$ if for every geometric point $\Spec K \to S$ the map
\[
\sigma(K) \colon \Mor_S(\Spec K, E) \hookrightarrow \Mor_S(\Spec K, X) \times \Mor_S(\Spec K, X) 
\] 
yields an equivalence relation on $K$-points of $X$. We say that $\sigma \colon E \to X \times_S X$  is finite if $\sigma_i$ are finite.
\end{definition}
See \cite[Definition 2]{kollar12} for another equivalent definition.

\begin{definition} Let $\sigma_1, \sigma_2 \colon E \rightrightarrows X$ be a set theoretic finite equivalence relation of separated algebraic spaces of finite type over a Noetherian base scheme $S$. We call $q \colon X \to Y$, for a separated algebraic space $Y$ of finite type over $S$, a \emph{categorical quotient}  if $q \circ \sigma_1 = q \circ \sigma_2$ and $q$ is universal with this property (in the category of separated algebraic spaces of finite type over $S$). We call $q$ a \emph{geometric quotient} if
\begin{itemize}
	\item it is a categorical quotient, 
	\item it is finite, and 
	\item for every geometric point $\Spec K \to S$, the fibres of $q_K \colon X_K(K) \to Y_K(K)$ are the $\sigma(E_K(K))$-equivalence classes of $X_K(K)$.
\end{itemize}
\end{definition}
Note that in contrast to Koll\'ar we do not require the spaces to be reduced in the definition of set theoretic finite equivalence relations. The following proposition shows that the assumption on being a categorical quotient can be replaced by saying that $\mcO_Y$ is the kernel of $\sigma_1^*-\sigma_2^*$. 

\begin{proposition}[{\cite[Lemma 17]{kollar12}}] \label{thm:kollar_quotients_exist} Let $X$ be a separated algebraic space of finite type over a Noetherian base scheme $S$, let $Y$ be an algebraic space over $S$, let $\pi \colon X \to Y$ be an integral morphism over $S$, and let $E \rightrightarrows X$ be a finite set theoretic equivalence relation over $Y$. Then the geometric quotient $X/E$ exists as a separated algebraic space of finite type over $S$.
\end{proposition}
\begin{proof}
Note that $X \to Y$ is automatically finite as $X$ is of finite type over $S$. We claim that $Z = \Spec_Y \ker(\sigma_1^* - \sigma_2^* \colon \pi_*\mcO_X \to \pi_*\mcO_E)$ is the geometric quotient, where the projection from $E$ to $Y$ is by abuse of notation denoted by $\pi$. Let $q \colon X \to Z$ be the induced map. By construction, $q$ is finite and $\mcO_Z \to q_*\mcO_X$ is injective, hence $Z$ is separated and of finite type over $S$ by Lemma \ref{lem:fin_gen_when_finite}. Moreover, $q$ is a quotient on geometric points (by the same argument as in \cite[Lemma 17]{kollar12}) and $q \circ \sigma_1 = q \circ \sigma_2$. Thus, it is enough to show that it is a categorical quotient.

To this end, consider a map $f \colon X \to W$ to a separated algebraic space of finite type over $S$ which equalises $\sigma$. Let $Z^*$ be the image of $(q,f) \colon X \to Z \times_S W$. Since $Z \times_S W$ is separated and of finite type over $S$, so is $Z^*$. It is enough to show that the induced map $h \colon Z^* \to Z$ is an isomorphism. Since $q \colon X \to Z$ is a quotient on geometric points and the induced map $q^* \colon X \to Z^*$ equalises $\sigma$, we get that $h \colon Z^* \to Z$ is a bijection on geometric points. By Remark \ref{rem:scheme_facts}(\ref{itm:08AJ}), $h \colon Z^* \to Z$ is proper, and so by Remark \ref{rem:scheme_facts}(\ref{itm:0A4X}) it is a finite universal homeomorphism. By construction, $h_*\mcO_{Z^*} \to \ker(\sigma_1^* - \sigma_2^* \colon q_*\mcO_X \to q_*\mcO_E)=\mcO_Z$ is an injection of $\mcO_Z$-sheaves, thus $h_* \mcO_{Z^*} = \mcO_Z$ and $h \colon Z^* \to Z$ is an isomorphism. \qedhere  
\end{proof}


\subsection{Quotients by group schemes}
The following definitions are taken from \cite[Definition 2.7]{kollar97}.
\begin{definition} Let $X$ be an algebraic space over a Noetherian scheme $S$, and let $G$ be a group scheme over $S$ acting on $X$. We say that $q \colon X \to Z$ is a \emph{topological quotient} of $X$ by $G$ if $q$ is a $G$-morphism (with $Z$ admitting a trivial action), it is locally of finite type, it is universally submersive, and it is a set quotient on the level of geometric points. If in addition $\mcO_Z = (q_*\mcO_X)^G$, then we call $q$ a $\emph{geometric quotient}$.
\end{definition}
We say that an action of $G$ on $X$ is proper if $\psi_X \colon G \times_S X \xrightarrow{(m_X,p_2)} X \times_S X$ is proper, where $m_X \colon G \times_S X \to X$ is the morphism underlying the action of $G$, and $p_2 \colon G\times_S X \to X$ is the projection on the second factor. Since $G$ is affine, this condition ensures that the stabilisers are finite.  

We state an analogue of Theorem \ref{thm:kollar_quotients_exist}.
\begin{theorem}[{\cite[Theorem 3.13]{kollar97}}] \label{thm:kollar_quotients_group_schemes_exist} Let $G$ be an affine algebraic group scheme, flat and locally of finite type over $S$. Let $m \colon G \times X \to X$ be a proper $G$-action on an algebraic space $X$ over $S$. Let $f \colon X \to Z$ be a topological quotient. Then a geometric quotient $g \colon X \to X/G$ exists and is defined by the formula $X/G := \Spec_Z(f_*\mcO_X)^G$. Moreover, the induced map $X/G \to Z$ is a finite universal homeomorphism.
\end{theorem}  

\begin{remark} \label{remark:group_action_facts} With notation as above, suppose that $X$ is a separated algebraic space and $G$ is an affine algebraic group scheme, flat and of finite type over $S$ and which acts properly on $X$. Note that $G$ is of finite presentation (\stacksproj{01TX}) and universally open over $S$ (\stacksproj{01UA}). Let $q \colon X \to Z$ be a finite type topological quotient. Then
\begin{enumerate}
\item $m \colon G \times X \to X$ is affine and of finite type,
\item $q$ is affine and $Z$ is separated,
\item if $X$ is of finite type over $S$, then so is $Z$,
\item a geometric quotient is automatically a categorical quotient,
\item if $f \colon X \to Y$ is a finite surjective $G$-morphism of separated algebraic spaces of finite type admitting a proper $G$-action and the geometric quotients $X/G$ and $Y/G$ exist, then the induced map $f_G \colon X/G \to Y/G$ is finite. Moreover, if $f$ is a finite universal homeomorphism, then so is $f_G$.
\end{enumerate}
The morphism $m \colon G \times X \to X$ may be identified with $p_2 \colon G \times X \to X$ via the isomorphism  $(p_1, m_X) \colon G \times X \to G \times X$, so (1) holds. The morphism $q$ is affine by \cite[Theorem 3.12]{kollar97}. The quotient $Z$ is separated by \cite[Proposition 2.9]{kollar97} and of finite type over $S$ (provided so is $X$) by \cite[Theorem 3.12]{kollar97}. A geometric quotient is categorical by \cite[Corollary 2.15]{kollar97}. Last, for (5), consider the following diagram
\begin{center}
\begin{tikzcd}
X \arrow{r}{f} \arrow{d}{q} & Y \arrow{d}{q_Y} \\
X/G \arrow{r}{f_G} & Y/G.
\end{tikzcd}
\end{center}
By the above, both $X/G$ and $Y/G$ are separated and of finite type, hence so is $f_G$. Let $Z\subseteq X/G$ be a closed subset. Then $f(q^{-1}(Z)) = q_Y^{-1}(f_G(Z))$ is closed, and hence so is $f_G(Z)$ as $q_Y$ is submersive. The same holds after any base change by a morphism to $Y/G$, thus $f_G$ is universally closed, and so proper. By Remark \ref{rem:scheme_facts}(\ref{itm:0A4X}), it is finite and the last assertion can be checked on geometric points.

\end{remark}


\subsection{Pushouts of universal homeomorphisms}
In this subsection we discuss some preliminary results on pushouts of universal homeomorphisms. The case of pushouts of thickenings by affine morphisms is well understood and described in \stacksproj{07RT and 07VX}. 

\begin{definition}[{cf.\ \cite[Section 8]{kollar97}}] \label{definition:geo_pushout} Consider the following commutative diagram of schemes or algebraic spaces over a scheme $S$
\begin{center}
\begin{tikzcd}
X \arrow{d}{f}  & Y \arrow{l}{p} \arrow{d}{g} \\
X' & Y', \arrow{l}{q}
\end{tikzcd}
\end{center}
where $Y \xrightarrow{p} X$ is representable, quasi-compact, and quasi-separated, and $Y \xrightarrow{g} Y'$ is a representable universal homeomorphism. We say that this diagram is a \emph{topological pushout square} if $f$ is a representable universal homeomorphism and a \emph{geometric pushout square} if in addition
\[
\mcO_{X'} = f_* \mcO_X \times_{(f \circ p)_* \mcO_Y} q_*\mcO_{Y'}.
\]
We write $X' = X \sqcup_Y Y'$ and say that $X'$ is a \emph{topological} or a \emph{geometric pushout}. If $X$ is a scheme, then so is $X'$ by Remark \ref{rem:scheme_facts}(\ref{itm:07VV}).

\end{definition}
The assumption on the representability of $p$ may not be necessary. In any case, we are mostly interested in the case of $p$ being affine or a morphism from a scheme $Y$ to an algebraic space $X$.
\begin{lemma} \label{lem:0et0} Let $A \to B \leftarrow B'$ be maps of rings such that $B' \to B$ is a universal homeomorphism and let $A' = A \times_B B'$. Suppose that $A' \to A$ is a universal homeomorphism. Then $\Spec A'$ is a geometric pushout of $\Spec A \leftarrow \Spec B \to \Spec B'$.
\end{lemma}
\begin{proof}
This follows by the same proof as \stacksproj{0ET0} (see also \stacksproj{07RT and 01Z8}).
\end{proof} 
\begin{remark} \label{remark:geo_pushouts_affine_and_finite} Consider a  topological pushout square of schemes or algebraic spaces as above. Then 
\begin{enumerate}
	\item The morphism $q$ is representable, quasi-compact, and quasi-separated. 
	\item If $p$ is separated, affine, universally closed, or integral, then so is $q$.
	\item If $X$ is quasi-compact, quasi-separated, or separated, then so is $X'$.
	\item If the pushout is geometric and $g^* \colon \mcO_{Y'} \to g_*\mcO_{Y}$ is injective, then $f^* \colon \mcO_{X'} \to f_*\mcO_X$ is injective as well.
\end{enumerate}
To prove (1) and (2), we can assume that $X'$ is an affine scheme, and so that $X$ is affine and $Y$ is a quasi-compact quasi-separated scheme. Then $Y'$ is also a scheme (Remark \ref{rem:scheme_facts}(\ref{itm:07VV})), it is quasi-compact (Remark \ref{rem:scheme_facts}(\ref{itm:qc})), and quasi-separated (Remark \ref{rem:scheme_facts}(\ref{itm:09MQ})). Thus $q$ is representable, quasi-compact, and quasi-separated. If $p$ is separated or affine, then we can assume that $Y$ is separated or affine, respectively, and then so is $Y'$ (Remark \ref{rem:scheme_facts}(\ref{itm:09MQ})(\ref{itm:05YU})). Thus $q$ is separated or affine, respectively. If $p$ is universally closed, then so is $q$ by Remark \ref{rem:scheme_facts}(\ref{itm:03GN}) applied to $q \circ g$. Since being integral is equivalent to being affine and universally closed (Remark \ref{rem:scheme_facts}(\ref{itm:01WM})), the integrality of $p$ implies the integrality of $q$. 

The quasi-compactness, quasi-separatedness, or separatedness of $X'$, provided $X$ has these properties, respectively, follows, as above, from Remark \ref{rem:scheme_facts}(\ref{itm:09MQ})(\ref{itm:qc}). The injectivity of $f^*$ provided the injectivity of $g^*$ follows by definition. 

\end{remark}


\begin{lemma}[{cf.\ \cite[(8.1.3)]{kollar97}}] \label{lemma:technical_pushouts} 
Let $X \leftarrow Y \rightarrow Y'$ be a diagram of schemes (algebraic spaces, resp.) satisfying the assumptions of Definition \ref{definition:geo_pushout} and which admits a topological pushout $Z$. Then the geometric pushout $X' \coloneq X \sqcup_Y Y'$ exists as a scheme (an algebraic space, resp.). Moreover, the induced map $X' \to Z$ is a representable universal homeomorphism.
\end{lemma}
\begin{proof}
Define $X' \coloneq \Spec_Z \big((f_Z)_*\mcO_{X} \times_{(f_Z \circ p)_*\mcO_{Y}} (q_Z)_*\mcO_{Y'} \big)$ sitting inside

\begin{center}
\begin{tikzcd}
  & X \arrow{d}{f} \arrow[bend right = 15]{ldd}[swap]{f_Z}  & Y \arrow{l}{p} \arrow{d}{g} \\
  & X' \arrow[dashed]{ld} & Y'.  \arrow{l}{q} \arrow[bend left = 15]{lld}{q_Z} \\
Z & & 
\end{tikzcd}
\end{center}
Here, we used quasi-compactness and quasi-separatedness of morphisms (Remark \ref{remark:geo_pushouts_affine_and_finite}(1)) to guarantee that the pushforwards of structure sheaves are quasi-coherent (see \stacksproj{03M9}). 

Now, by means of \'etale base change, we can assume that the spaces in question are schemes. By construction, $\ker(\mcO_{X'} \to f_*\mcO_{X}) \simeq \ker(q_*\mcO_{Y'} \to q_*g_*\mcO_Y)$, and so the kernel of $\mcO_X' \to f_*\mcO_X$ is locally nilpotent, and, in particular, $X \to X'$ is dominant. Moreover, $X' \to Z$ is separated as it is affine. Thus, both $X \to X'$ and $X' \to Z$ are universal homeomorphisms by Lemma \ref{lemma:factors_of_universal_homeomorphism}, and so $X'$ is a geometric pushout.
\end{proof}

\begin{remark} \label{remark:pushout_need_not_be_noetherian} Even if $X \leftarrow Y \rightarrow Y'$ are of finite type over a field $k$, the geometric pushout need not be Noetherian (see \cite[Example 8.5]{kollar97}). A pertinent example which is relevant to us is the following pushout diagram:
\begin{center}
\begin{tikzcd}
\Spec \mbZ \arrow{d} &  \Spec \mbQ \arrow{l} \arrow{d} \\
\Spec \mbZ \oplus x\mbQ & \Spec \mbQ[x]/x^2 \arrow{l}.
\end{tikzcd}
\end{center}
\end{remark}

To rectify the problem laid down in the above remark, we use Noetherian approximation.
\begin{lemma} \label{lem:fin_gen_pushouts} 
Let $X \xleftarrow{p} Y \xrightarrow{g} Y'$ be a diagram of schemes (or algebraic spaces) over a Noetherian base scheme $S$, satisfying the assumptions of Definition \ref{definition:geo_pushout} and admitting a topological pushout $X'$. Assume that $X$ is of finite type over $S$. Then there exists a topological pushout $X'_{\rm top}$ of $X \xleftarrow{} Y \xrightarrow{} Y'$, which is of finite type over $S$. Moreover, if $X$ is proper over $S$, then so is $X'_{\rm top}$.
\end{lemma}
\begin{proof}
By Lemma \ref{lemma:technical_pushouts} we can assume that $X'$ is a geometric pushout. Note that $X$, $X'$, and $S$ are quasi-compact and quasi-separated (cf.\ Remark \ref{remark:geo_pushouts_affine_and_finite}(3)). Thus, we can apply \stacksproj{09MV} (\stacksproj{09NR}, resp.) to get an inverse system of schemes (algebraic spaces, resp.) $X'_i$, of finite type over $S$, over a directed set $I$ with affine transition maps such that $X' = \varprojlim X'_i$. 

Since $f \colon X \to X'$ is a representable universal homeomorphism, the induced map $f(X) \to X'$ is a thickening and $f \colon X \to f(X)$ is a representable universal homeomorphism (cf.\ Remark \ref{rem:scheme_facts}(\ref{itm:04DZ})). Moreover, $f(X)$ is of finite type over $S$ by Lemma \ref{lem:Eakin-Nagata} as $\mcO_{f(X)} \to f_*\mcO_X$ is injective by definition. Thus, by \stacksproj{081B} (\stacksproj{0828}, resp.), there exists $X'_i$ such that the composition  $f(X) \to X' \to X'_i$ is a closed immersion. By replacing $X'_i$ by the image of $X'$ in it, we can assume that $\mcO_{X'_i} \to \mcO_{X'}$ is injective, and hence the kernel of $\mcO_{X'_i} \to \mcO_{f(X)}$ is locally nilpotent. Therefore, $f(X) \to X'_i$ is a thickening, thus $X \to X'_i$ is a representable universal homeomorphism and $X'_{\rm top} \coloneq X'_i$ is a topological pushout of  $X \xleftarrow{p} Y \xrightarrow{g} Y'$. 

To show the last statement, we note that the properness of $X$ implies that $X'_{\rm top}$ is separated over $S$ (Remark \ref{rem:scheme_facts}(\ref{itm:09MQ})), and hence proper by Remark \ref{rem:scheme_facts}(\ref{itm:03GN}).
\end{proof}

\begin{lemma}[{cf.\ \cite[Lemma 8.2]{kollar97}}] \label{lem:properties_of_geo_pushouts}
A base change of a geometric pushout square by a flat morphism is a geometric pushout square. 
\end{lemma}
\begin{proof}
A geometric pushout square as in Definition \ref{definition:geo_pushout} is uniquely determined by the following exact sequence
\[
0 \to \mcO_{X'} \to f_*\mcO_X \oplus q_* \mcO_{Y'} \to (f \circ p)_* \mcO_Y
\] 
and the fact that $X \to X'$ is a representable universal homeomorphism. These properties are preserved under flat base change. \qedhere 

\end{proof}

Further, we study \'etale morphisms under geometric pushouts.  
\begin{proposition}[{cf.\ \cite[Lemma 44]{kollar12}}] \label{prop:pushouts_induce_etale_maps}
Let 
\begin{center}
\begin{tikzcd}
X_1 \arrow{d} & Y_1 \arrow{l}[swap]{p_1} \arrow{r}{g_1} \arrow{d} & Y_1' \arrow{d} \\
X_2 & Y_2 \arrow{l}[swap]{p_2} \arrow{r}{g_2} & Y_2'
\end{tikzcd}
\end{center}
be a commutative diagram of schemes such that both squares are Cartesian and $X_i \leftarrow Y_i \rightarrow Y'_i$ satisfy the assumptions of Definition \ref{definition:geo_pushout} for $i \in \{1,2\}$. Further, suppose that  the vertical maps are \'etale and the geometric pushout $X_2'$ of the second row {$(p_2,g_2)$} exists. Then a geometric pushout $X_1'$ of the first row {$(p_1,g_1)$} exists and the induced map $X_1' \to X_2'$ is \'etale. 
\end{proposition} 
\begin{proof}
By Remark \ref{rem:scheme_facts}(\ref{itm:04DZ}) applied to the universal homeomorphism $X_2 \to X_2'$ we can find a scheme $X'_1$ and an \'etale morphism $X'_1 \to X'_2$ the pullback of which is $X_1 \to X_2$.  Moreover, the pullback of $X'_1 \to X_2'$ to $Y_2'$ agrees with $Y'_1 \to Y'_2$ by applying Remark \ref{rem:scheme_facts}(\ref{itm:04DZ}) to the universal homeomorphism $Y_2 \to Y_2'$. In particular, we get the following commutative diagram
\begin{center}
\begin{tikzcd}
X_1 \arrow[dashed]{d} \arrow{r}     & X_2 \arrow{d}  & Y_2 \arrow{l} \arrow{d}   &  Y_1 \arrow[bend right = 30]{lll} \arrow{l} \arrow{d} \\
X'_1 \arrow[dashed]{r}     & X'_2 & Y'_2 \arrow{l}  & Y'_1, \arrow{l} \arrow[bend left = 30, dashed]{lll} 
\end{tikzcd}
\end{center}
where the bigger square is a pull-back of the smaller square via $X'_1 \to X'_2$. By Lemma \ref{lem:properties_of_geo_pushouts} the bigger square is thus a geometric pushout. \qedhere

\end{proof}
\begin{lemma} \label{lemma:geometric_pushout_is_a_categorical_pushout} Let $X'$ be a geometric pushout of a diagram $X \leftarrow Y \rightarrow Y'$ of schemes (algebraic spaces, resp.) satisfying the assumptions of Definition \ref{definition:geo_pushout}. Then $X'$ is a categorical pushout in the category of schemes (algebraic spaces, resp.).
\end{lemma} 
\begin{proof}
Since algebraic spaces are quotients of schemes, one can reduce to the case of $X$, $X'$, $Y$, and $Y'$ being schemes (see the end of \stacksproj{07VX}). By \stacksproj{07SX}, it is enough to show that $X'$ is a pushout in the category of schemes (assumptions (3) and (4) are satisfied by Lemma \ref{lem:properties_of_geo_pushouts} and Proposition \ref{prop:pushouts_induce_etale_maps}, respectively). 

We argue as in \stacksproj{0ET0}. Suppose there is a scheme $Z$ and maps $f_Z \colon X \to Z$ and $q_Z \colon Y' \to Z$ agreeing on $Y$. We can define $h \colon X' \to Z$ as equal to $f_Z$ on the level of topological spaces. Moreover, $h$ is a map of ringed spaces via $\mcO_Z \to (f_Z)_*\mcO_X \times_{(f_Z \circ p)_*\mcO_Y} (q_Z)_* \mcO_{Y'} = h_*\mcO_{X'}$. In fact, it is a map of locally ringed spaces (and hence of schemes) as $f \colon X \to X'$ is a universal homeomorphism and $f_Z$ is a map of schemes (cf.\ the last paragraph of \stacksproj{0ET0}).
\end{proof}


Last, we prove that it is enough to construct geometric pushouts locally. 

\begin{lemma} \label{lem:construct_pushout_locally}
Let $X \xleftarrow{p} Y \xrightarrow{g} Y'$ be a diagram of schemes (algebraic spaces, resp.) satisfying the assumptions of Definition \ref{definition:geo_pushout}. Then a geometric pushout of this diagram exists as a scheme (an algebraic space, resp.), if and only if it exists after pulling back by every open immersion (\'etale morphism, resp.) $U \to X$ with $U$ an affine scheme. 
\end{lemma}
Here $U \xleftarrow{p} V \xrightarrow{g} V'$ is a \emph{pullback} of $X \xleftarrow{p} Y \xrightarrow{g} Y'$ by an \'etale morphism $U \to X$ if $V = U \times_X Y$ and $V' \to Y'$ is the unique \'etale map with the pullback via $g$ being $V \to Y$ (see Remark \ref{rem:scheme_facts}(\ref{itm:04DZ})). If $U$ is a scheme, then so are $V$ and $V'$ (Remark \ref{rem:scheme_facts}(\ref{itm:07VV})). If $U \to X$ is an open immersion, then  $V = p^{-1}(U)$ and $V' = g(V)$. Note that $U \leftarrow V \to V'$ satisfies the assumptions of Definition \ref{definition:geo_pushout}.
\begin{proof}
If a geometric pushout of $X \leftarrow Y \to Y'$ exists, then it exists after the pullbacks by Proposition \ref{prop:pushouts_induce_etale_maps}.
As for the implication in the other direction, we first deal with the case of schemes arguing as in \stacksproj{07RT}. Let 
\begin{center}
\begin{tikzcd}
X \arrow{d}{f} & \arrow{l}{p} \arrow{d}{g} Y \\
X' & \arrow{l}{q} Y'
\end{tikzcd}
\end{center}
be a push-out diagram of topological spaces such that $X'=X$, $f= \mathrm{id}$, and $q \colon Y' = Y \xrightarrow{p} X = X'$ is the natural map induced by $p$ {(that is, topologically, $q = g^{-1} \circ p$)}. 

We make $f$ into a map of ringed spaces by setting
\[
\mcO_{X'} \coloneq f_*\mcO_X \times_{(f \circ p)_* \mcO_{Y}} q_* \mcO_{Y'}.
\]
The fact that $f$ is a map of schemes and is a universal homeomorphism can be checked locally on $X$ and hence follows from the assumptions and Lemma \ref{lem:0et0}.\\

Now, we move to the case of algebraic spaces (cf.\ \stacksproj{07VX}). Pick a surjective \'etale map $U \to X$ with $U$ a scheme, and construct pushouts $U'$ and $E'$ of the pullbacks of $X \xleftarrow{p} Y \xrightarrow{g} Y'$ by $U\to X$ and $E := U\times_X U \rightrightarrows U \to X$, respectively (they exist by the above paragraph). Then the maps $\sigma_1', \sigma_2' \colon E' \rightrightarrows U'$ are \'etale by Proposition \ref{prop:pushouts_induce_etale_maps}. Moreover, they induce an equivalence relation on $U'$; indeed, $E \to E'$ is a universal homeomorphism, and so by \stacksproj{0DT7} it is enough to construct the identity $e \colon U' \to E'$, the inversion $i \colon E' \to E'$, and the composition map $c' \colon E' \times_{\sigma_2',U', \sigma_1'} E' \to E'$ which follow by functoriality of pushouts (we leave details to the reader). Thus, we can take a quotient $X' \coloneq U' / E'$ as an algebraic space (\stacksproj{02WW}) sitting inside the following diagram:
\begin{center}
\begin{tikzcd}
E \arrow[shift right = 2pt]{r} \arrow[shift left = 2pt]{r} \arrow{d} & U \arrow{r} \arrow{d} & X \arrow{d} \\
E' \arrow[shift right = 2pt]{r} \arrow[shift left = 2pt]{r} & U' \arrow{r} & X'.
\end{tikzcd} 
\end{center}
Since the left diagram is Cartesian, $X \to X'$ is injective (cf.\ \stacksproj{045Z}). We claim that the right diagram is Cartesian (and so $X \to X'$ is a representable universal homeomorphism). Indeed, the morphism $U \to U'$ factorises as
\[
U \to U' \times_{X'} X \to U',
\]
and so $U \to U' \times_{X'} X$ is universally injective. Since $X \to X'$ is injective, so is $U' \times_{X'} X \to U'$. Thus $U \to U' \times_{X'} X$ is surjective. Moreover, it is \'etale by \stacksproj{03FV} as $U'\times_{X'} X \to X$ and $U \to X$ are \'etale, and hence it is an isomorphism as $U \to U' \times_{X'} X$ is representable (see \stacksproj{02LC}). \qedhere
\end{proof}

\subsection{Generalised conductor squares} \label{ss:generalised-conductor-square}
\begin{definition} 
We call a commutative diagram
\begin{center}
\begin{tikzcd}
X \arrow{d}{f} & D \arrow{d}{g} \arrow{l}{i} \\
Y  & C \arrow{l}{j}
\end{tikzcd}
\end{center}
a \emph{{generalised} conductor square} when $f \colon X \to Y$ is a finite {surjective} map of reduced Noetherian schemes and $D \to C$ is the induced map of the conductors of $f$.
\end{definition}
{We define conductors affine locally exactly as in \cite[I.2.6]{weibel13}. Precisely, given a finite extension of rings $R \subseteq S$ we set $I = \{ s \in S \,\mid\, sS \subseteq R\}$. The ideals $I\subseteq R$ and $I = IS \subseteq S$ define the conductors $C$ and $D$. Note that $R \simeq S \times_{S/I} R/I$; this is an example of a Milnor square.

When $f$ is finite of degree greater than one over each irreducible component, then $C = Y$ and $D=X$. We called the above diagram a generalised conductor square, as conductors are often defined only when $f$ is birational (that is, an isomorphism over a \emph{dense} open subset), cf.\ \cite[Definition 2.25]{ct17}. Later on, we consider finite universal homeomorphisms $f \colon X \to Y$ which are isomorphisms over an open subset only (and so, on some irreducible components, may be equal to, say, Frobenius).  }


\begin{lemma} \label{lemma:Milnor_Pic} Consider a {generalised} conductor square as above. Then the diagram 
\begin{center}
\begin{tikzcd}
\PicS_X \arrow{r}{i^*} & \PicS_D  \\
\PicS_Y \arrow{u}{f^*} \arrow{r}{j^*} & \PicS_C  \arrow{u}{g^*}
\end{tikzcd}
\end{center}
is Cartesian in the $2$-category of groupoids.
\end{lemma}
\noindent This stipulates that there exists a functorial one-to-one correspondence between line bundles $L_Y$ on $Y$ and triples $(L_X,L_C, \phi)$ where $L_X$ and $L_C$ are line bundles on $X$ and $C$, respectively, and $\phi \colon g^*L_C \to i^*L_X$ is an isomorphism.
\begin{proof}
Given a line bundle $L$ on $Y$, we get an induced triple $(f^*L, j^*L, \phi)$, where $\phi \colon g^*j^*L  \xrightarrow{\simeq} i^*f^*L$. 

In the opposite direction, let $(L_X,L_C, \phi)$ be a triple as above and set $L_D := i^*L_X$. Therewith one can define a sheaf $L_Y$ on $Y$ by the formula
\[
L_Y \coloneq \mathrm{ker}(f_*L_X \times j_*L_C \xrightarrow{i^* - \phi \circ g^*} (f \circ i)_*L_D).
\]

To conclude the proof, we need to verify two things. First, that $L_Y$ is a line bundle. Second, that given a line bundle $L$ on $Y$ and an induced triple $(f^*L, j^*L, \phi)$, the natural map $L \to L_Y$ of sheaves to the induced line bundle on $Y$ is an isomorphism. Both statements can be verified locally, and hence the proof follows by \cite[Milnor Patching Theorem 2.7]{weibel13} (or \cite[Tag 0D2J]{stacks-project}) as {generalised} conductor squares of affine schemes are Milnor squares and finite rank one projective modules are line bundles (\stacksproj{00NX}).
\end{proof}

\section{Multiplicative perfection in mixed characteristic}

Throughout this section, we fix a prime number $p>0$ and work over the base ring $\Zpp$.

\subsection{Multiplicative perfection} The key advantage of working in positive characteristic is the existence of the Frobenius morphism. In mixed characteristic we shall approximate it by raising to a $p^n$-th power for big $n>0$.
\begin{definition} Let $A$ be a ring over $\mbZ_{(p)}$. We call the commutative monoid
\[
A^{\perf} \coloneq \varinjlim_{x \mapsto x^p} A,
\]
the (multiplicative) \emph{perfection} of $A$.
\end{definition}
Note that the multiplicative perfection does not preserve the additive structure. 
\begin{remark} The natural map $A \to A^{\perf}$ induces an inclusion $A/{\sim} \hookrightarrow A^{\perf}$ of monoids,  where $a \sim b$ if and only if $a^{p^n} = b^{p^n}$ for some $n\gg 0$.
\end{remark}
Further, note that for any rings $A$, $B$, $C$ over $\mbZ_{(p)}$ the {natural morphism of commutative monoids}
\[
(A \times_B C)^{\perf} \to A^{\perf} \times_{B^{\perf}} C^{\perf}
\]
{is a bijection of sets (and so an isomorphism of commutative monoids). This can be easily checked by hand or by recalling that filtered colimits commute with finite products (\cite[Tag 002W]{stacks-project}).}


\begin{definition} \label{def:multiplicative-perfection-of-sheaves} Let $L$ be a line bundle on a scheme $X$ over $\mbZ_{(p)}$. We call the {sheaf} of sets
\[
L^{\perf} \coloneq \varinjlim_{\phi_n} L^{p^n},
\]
where $\phi_n \colon L^{p^n} \to L^{p^{n+1}}$ with $\phi_n(x)=x^p$, the \emph{perfection} of $L$.
\end{definition}
\noindent {Explicitly, $L^{\perf}$ is the sheafification of the colimit taken in the category of presheaves.}

If $\Spec A = U \subseteq X$ is an affine subscheme such that $L|_U \simeq \mcO_U$, then we get a sequence of compatible isomorphisms $(L^{p^n})|_U \simeq \mcO_U$ for every $n \geq 0$, thus  $L^{\perf}(U) \simeq A^{\perf}$. 

Define
\[
H^0(X, L)^{\perf} := \varinjlim_{\phi_n} H^0(X, L^{p^n}),
\]
where $\phi_n \colon H^0(X,L^{p^n}) \to H^0(X,L^{p^{n+1}})$ with $\phi_n(x)=x^p$.
When $X$ is quasi-compact {and quasi-separated,} $H^0(X, L^{\perf}) = H^0(X, L)^{\perf}$ {by \cite[Tag 009F(4) and Tag 0069(3)]{stacks-project}}. 

\subsection{Infinitesimal site up to perfection}

The following lemma is vital in the proofs of the main results of this section. 

\begin{lemma} \label{lem:univ_homeo_iso_generic} Let $f \colon X \to Y$ be an affine morphism of schemes over $\Spec \Zpp$ such that $f|_{X_{\mbQ}} \colon X_{\mbQ} \to Y_{\mbQ}$ is an isomorphism. Then $f^* \colon \mcO^{\perf}_Y \to f_*\mcO^{\perf}_X$ is an isomorphism if and only if $f$ is a universal homeomorphism. 
\end{lemma}

The key to the results of this article is the local injectivity of $f^*$ as it allows for gluing sections and lifting them globally under thickenings.  Note that the other parts of the lemma have been shown in \cite[Lemma 8.7]{kollar97} (see also \stacksproj{0CNF}).
\begin{proof}
If $f^* \colon \mcO^{\perf}_Y \to \mcO^{\perf}_X$ is an isomorphism, then $f$ is a universal homeomorphism as well by \stacksproj{0CNF}. Thus, it is enough to show the converse. For the convenience of the reader, we also show local surjectivity of $f^*$.

Since $f$ is affine, the lemma can be reduced to showing that $\pi^{\perf} \colon B^{\perf} \to A^{\perf}$ is an isomorphism, when $\pi \colon B \to A$ is a universal homeomorphism such that the localisation $\pi_{\mbQ} \colon B[\frac{1}{p}] \to A[\frac{1}{p}]$ is an isomorphism.

Pick any element $a \in A$. Since $\pi$ is a universal homeomorphism, so is its reduction $\pi_p \colon B/p \to A/p$ modulo $p$. Thus, Proposition \ref{prop:univ_homeo_in_char_p} implies that 
\[
a^{p^l} = \pi(b) + pt
\]
for some $l \gg 0$, $b\in B$, and $t \in A$. 

As $\pi$ is integral, the $B$-subalgebra $A^0 \subseteq A$ generated by $t$ is a finite $B$-module. Given that $\pi_{\mbQ}$ is an isomorphism, we get $p^nA^0 \subseteq \pi(B)$ for some $n>0$ and hence $p^{n}t^i \in \pi(B)$ for every $i \geq 0$. 

Write
\[
a^{p^{k+l}} = \pi(b)^{p^k} + \sum_{i=1}^{p^{k}} {p^{k} \choose i} \pi(b)^{p^{k}-i} (pt)^i,
\]
for $k \gg 0$. Since {$p^{n} \mid p^i{p^{k} \choose i}$ for every $0 \leq i \leq p^k$ and $k \gg 0$}, the right hand side is contained in $\pi(B)$, and hence so is $a^{p^{k+l}}$. In particular, $\pi^{\perf}$ is surjective.\\

Now, assume that there exist $b, b' \in B$ satisfying $\pi(b) = \pi(b')$. Write $b = b' + s$ for some $s \in B$. Since $\pi_{\mbQ}$ is an isomorphism and $\pi(s)=0$, there exists $n>0$ such that $p^ns=0$. Since $\pi_p$ is a universal homeomorphism, Proposition \ref{prop:univ_homeo_in_char_p} implies that $s^{p^k}=pt$ for some $k>0$ and $t \in B$. In particular, $s^{np^k+1} = (pt)^ns = 0$ and we get
\[
b^{p^m} = (b')^{p^m} + \sum_{i=1}^{np^k} {p^m \choose i}(b')^{p^m-i}s^i = (b')^{p^m}
\] 
for $m \gg 0$. Here we used that $p^n \mid {p^m \choose i}$ for $1 \leq i \leq np^k$ and $m \gg 0$. As a consequence $\pi^{\perf}$ is injective which concludes the proof.
\end{proof}


Now, we can prove Lemma \ref{lemma:main}.
\begin{proof}[Proof of Lemma \ref{lemma:main}]
Note that $X_{\mbQ} \to X$ is quasi-compact and quasi-separated, as so is $\Spec \mbQ \to \Spec \mbZ_{(p)}$. Thus, by Lemma \ref{lemma:technical_pushouts}, there exists a pushout scheme $Z \coloneq X \sqcup_{X_{\mbQ}} Y_{\mbQ}$ sitting inside the following commutative diagram
\begin{center}
\begin{tikzcd}
Y &  &  \\
  & Z \arrow[dashed]{lu}{h}     & X \arrow[dashed]{l}{g} \arrow[bend right = 25]{llu}{f}  \\
  & Y_{\mbQ} \arrow[dashed]{u} \arrow[bend left = 25]{luu} & X_{\mbQ}, \arrow{l} \arrow{u}
\end{tikzcd}
\end{center}
with $g$ and $h$ being universal homeomorphisms. By construction, $h|_{Z_{\mbQ}} \colon Z_{\mbQ} \to Y_{\mbQ}$ is an isomorphism, and so Lemma \ref{lem:univ_homeo_iso_generic} implies that
$h^* \colon \mcO^{\perf}_Y \to h_*\mcO^{\perf}_{Z}$
is an isomorphism as well. We can conclude the proof as  
\[
\mcO^{\perf}_{Z} = \mcO^{\perf}_{Y_{\mbQ}} \times_{\mcO^{\perf}_{X_{\mbQ}}} \mcO^{\perf}_{X}. \qedhere
\]
\end{proof}

The following proposition is a direct consequence of {Lemma \ref{lemma:main}}.
\begin{proposition} \label{prop:sections_on_thickenings}  Let $f \colon X \to Y$ be a universal homeomorphism of quasi-compact quasi-separated schemes over $\Zpp$, and let $L$ be a line bundle on $Y$. Then the following diagram is Cartesian:
\begin{center}
\begin{tikzcd}
H^0(Y,L)^{\perf} \arrow{r}{f^*} \arrow{d} & H^0(X, f^*L)^{\perf} \arrow{d} \\
H^0(Y_{\mbQ}, L|_{Y_{\mbQ}})^{\perf} \arrow{r}  &H^0(X_{\mbQ}, f^*L|_{X_{\mbQ}})^{\perf}.
\end{tikzcd}
\end{center}
\end{proposition}
\begin{proof}
By Lemma \ref{lemma:main}, we get the following Cartesian diagram:
\begin{center}
\begin{tikzcd}
L^{\perf} \arrow{r}{f^*} \arrow{d} & (f^*L)^{\perf} \arrow{d} \\
(L|_{Y_{\mbQ}})^{\perf} \arrow{r} &  (f^*L|_{X_{\mbQ}})^{\perf},
\end{tikzcd}
\end{center}
Now, by applying $H^0$ to this diagram, we can conclude the proof.
\end{proof}

\subsection{Descending line bundles}
The goal of this subsection is to show Theorem \ref{thm:univ_homeo_and_Pic_intro}. Here, $\PicS_X$ denotes the groupoid of line bundles on $X$, and $\PicS_X[\frac{1}{p}]$ denotes the groupoid of line bundles on $X$ up to inverting $p$. {Informally, this is a groupoid of line bundles and their ``formal $p^n$-th roots''. Precisely, the objects of the category $\PicS_X[\frac{1}{p}]$ are pairs $(L,n)$ consisting of a line bundle $L$ on $X$ and a number $n\in \mbN$, and the morphisms being isomorphisms up to $p^N$-th power, that is $\Hom((\mcL,n), (\mcL',n')) = \varinjlim_N \mathrm{Isom}(\mcL^{p^{N-n}},\mcL'^{p^{N-n'}})$; cf.\ \cite[Tag 0EXA]{stacks-project}.} 

\begin{proof}[Proof of Theorem \ref{thm:univ_homeo_and_Pic_intro}]
We proceed by Noetherian induction on $X$.
\vspace{0.5em}

\textbf{Step 1.} The theorem holds when $X$ and $Y$ are defined over $\mathbb{F}_p$. Indeed, we have then that $\mcO^*_{Y}[\frac{1}{p}] \to \mcO^*_{X}[\frac{1}{p}]$ is an isomorphism (cf.\ the proof of \cite[Lemma 2.1]{ct17}), and so $\PicS_Y[\frac{1}{p}] \to \PicS_X[\frac{1}{p}]$ is an isomorphism as well.\\

\textbf{Step 2.} The theorem holds when $f \colon X \to Y$ is a thickening {(that is, a surjective closed immersion)}. Indeed, by Lemma \ref{lemma:technical_pushouts} (or \stacksproj{07RT}), there exists a pushout scheme $Z \coloneq X \sqcup_{X_{\mbQ}} Y_{\mbQ}$ sitting inside the following commutative diagram
\begin{center}
\begin{tikzcd}
Y &  &  \\
  & Z \arrow[dashed]{lu}{h}     & X \arrow[dashed]{l}{g} \arrow[bend right = 25]{llu}{f}  \\
  & Y_{\mbQ} \arrow[dashed]{u} \arrow[bend left = 25]{luu} & X_{\mbQ}. \arrow{l} \arrow{u}
\end{tikzcd}
\end{center}
As in the proof of Lemma \ref{lemma:main}, we get that $h^* \colon \mcO^{\perf}_Y \to \mcO^{\perf}_{Z}$
is an isomorphism. Since both $g$ and $f$ are thickenings, $h$ induces an isomorphism of reductions $Z_{\red} \to Y_{\red}$. Invertibility of sections does not depend on the infinitesimal structure, and so $h^* \colon \mcO^{*}_Y[\frac{1}{p}] \to \mcO^{*}_{Z}[\frac{1}{p}]$ is an isomorphism as well.
As a consequence, $h^* \colon \PicS_Y[\frac{1}{p}] \to \PicS_Z[\frac{1}{p}]$ is an isomorphism.

Since flat {finitely presented} coherent sheaves of rank one are line bundles (\stacksproj{00NX}), \stacksproj{08KU} (cf.\ Lemma \ref{lemma:Milnor_Pic}) implies
\[
\PicS_{Z} = \PicS_{Y_{\mbQ}} \times_{\PicS_{X_{\mbQ}}} \PicS_{X}.
\]
By inverting $p$, we can conclude Step 1.\\

\textbf{Step 3.} We reduce to the case when both $X$ and $Y$ are reduced. In this step we assume that the proposition is true for $f \colon X \to Y$ being replaced by its reduction $f^{\red} \colon X^{\red} \to Y^{\red}$. We have the following spacial commutative diagram:

\begin{center}
\begin{tikzcd}
&
\PicS_{Y^{\red}}
\ar[leftarrow]{dl}[swap, sloped, near start]{}
\ar[rightarrow]{dd}[near end]{}
& & \PicS_{X^{\red}}
\ar[rightarrow]{dd}{}
\ar[leftarrow]{ll}{}
\ar[leftarrow]{dl}[swap, sloped, near start]{}
\\
\PicS_{Y}

\ar[rightarrow]{dd}[swap]{}
& & \PicS_{X} \ar[crossing over, leftarrow]{ll}[near start]{}
\\
&
\PicS_{Y^{\red}_{\mbQ}}

\ar[leftarrow, sloped, near end]{dl}{}
& & \PicS_{X^{\red}_{\mbQ}} \ar[leftarrow, near start]{ll}{}
\ar[leftarrow]{dl}
\\
\PicS_{Y_{\mbQ}}

& & \PicS_{X_{\mbQ}} \ar[leftarrow]{ll}
\ar[crossing over, near start, leftarrow]{uu}{}
\end{tikzcd}
\end{center}
The left and the right facets are $2$-pullback squares up to inverting $p$ by Step~2 and the back one is a $2$-pullback square up to inverting $p$ by assumption. A composition of two $2$-pullback squares stays a $2$-pullback square (\stacksproj{02XD}), and so we have the following diagram
\begin{center}
\begin{tikzcd}
\PicS_Y \arrow{r} \arrow{d} & \PicS_{X} \arrow{r} \arrow{d} & \PicS_{X^{\red}} \arrow{d} \\
\PicS_{Y_{\mbQ}} \arrow{r} & \PicS_{X_{\mbQ}} \arrow{r} & \PicS_{X^{\red}_{\mbQ}},
\end{tikzcd}
\end{center}
in which the big square and the right square are $2$-pullback squares up to inverting $p$. Thus the left square is a $2$-pullback square up to inverting $p$ as well (\stacksproj{02XD}).\\

\textbf{Step 4.} We show the proposition under the assumption that $X$ and $Y$ are reduced and the proposition holds for every universal homeomorphism $f|_W \colon W \to f(W)$, where $W$ is a closed proper subset of $X$.\\

By Step 1, we may assume that $X_{\mbQ} \neq \emptyset$. Since $f \colon X \to Y$ is a finite {surjective} map of reduced schemes, it sits in the following {generalised} conductor square (cf.\ Subsection \ref{ss:generalised-conductor-square}, \cite[I.2.6]{weibel13}, \cite[Subsection 2.6.1]{ct17})
\begin{center}
\begin{tikzcd}
Y & \arrow{l}{f} X \\
C \arrow{u} & \arrow{l}{f_D} D, \arrow{u}
\end{tikzcd}
\end{center}
where $C$ and $D$ are conductors of $f$. Since $f \colon X \to Y$ is a finite universal homeomorphism, so is $f_D \colon D \to C$ (it is finite by Remark \ref{rem:scheme_facts}(\ref{itm:035D}) applied to $D \to C \to Y$). Note that $D$ is a strict closed subset of~$X$, { as $X_{\mbQ} \neq \emptyset$, and so $f$ is an isomorphism over an open subset of $Y$.}

As above, we can construct the following spatial diagram. 
\begin{center}
\begin{tikzcd}
&
\PicS_{C}
\ar[leftarrow]{dl}[swap, sloped, near start]{}
\ar[rightarrow]{dd}[near end]{}
& & \PicS_{D}
\ar[rightarrow]{dd}{}
\ar[leftarrow]{ll}{}
\ar[leftarrow]{dl}[swap, sloped, near start]{}
\\
\PicS_{Y}

\ar[rightarrow]{dd}[swap]{}
& & \PicS_{X} \ar[crossing over, leftarrow]{ll}[near start]{}
\\
&
\PicS_{C_{\mbQ}}

\ar[leftarrow, sloped, near end]{dl}{}
& & \PicS_{D_{\mbQ}} \ar[leftarrow, near start]{ll}{}
\ar[leftarrow]{dl}
\\
\PicS_{Y_{\mbQ}}

& & \PicS_{X_{\mbQ}} \ar[leftarrow]{ll}
\ar[crossing over, near start, leftarrow]{uu}{}
\end{tikzcd}
\end{center}
The top and the bottom facets are $2$-pullback squares by Lemma \ref{lemma:Milnor_Pic} and the back one is a $2$-pullback square up to inverting $p$ by the inductive assumption.  By the same argument as in Step 3, the front facet is a $2$-pullback square up to inverting $p$, too. \qedhere
\end{proof}
\begin{remark} An analogous argument shows that
\begin{center}
\begin{tikzcd}
\mcO^*_Y[\frac{1}{p}] \arrow{r}{f^*} \arrow{d} & \mcO^*_X[\frac{1}{p}] \arrow{d} \\
\mcO^*_{Y_{\mbQ}}[\frac{1}{p}] \arrow{r} & \mcO^*_{X_{\mbQ}}[\frac{1}{p}],
\end{tikzcd}
\end{center}
is a pullback square. Other types of functors with this property will be discussed in \cite{aemw}.
\end{remark}

\begin{corollary} \label{cor:pic_for_general_pushouts} Let $X'$ be a Noetherian topological pushout of a diagram $X \leftarrow Y \to Y'$ of Noetherian algebraic spaces over $\mathbb{Z}_{(p)}$ satisfying the assumptions of Definition \ref{definition:geo_pushout} and such that both $f \colon Y \to Y'$ {and $X \to X'$} are finite universal homeomorphisms. Then
\begin{center}
\begin{tikzcd}
\PicS_{X'}  \arrow{d}  \arrow{r}  &  \PicS_{X} \times_{\PicS_Y}  \PicS_{Y'}  \arrow{d} \\
\PicS_{X'_{\mbQ}}  \arrow{r} & \PicS_{X_\mbQ} \times_{\PicS_{Y_{\mbQ}}} \PicS_{Y'_{\mbQ}}, 
\end{tikzcd}
\end{center}
is a $2$-pullback square up to inverting $p$. 
\end{corollary}
 Inverting $p$ commutes with products in the above diagram. Using the language of 2-categories makes the statement and the proof of this result incomparably easier and cleaner.
\begin{proof}
Construct the following spatial diagram.
\begin{center}
\begin{tikzcd}
&
\PicS_{X}
\ar[leftarrow]{dl}[swap, sloped, near start]{}
\ar[rightarrow]{dd}[near end]{}
& & \PicS_{Y}
\ar[rightarrow]{dd}{}
\ar[leftarrow]{ll}{}
\ar[leftarrow]{dl}[swap, sloped, near start]{}
\\
{ \hphantom{aaa}\,\PicS_{X} \times_{\PicS_Y} \PicS_{Y'}}

\ar[rightarrow]{dd}[swap]{}
& & \PicS_{Y'} \ar[crossing over, leftarrow]{ll}[near start]{}
\\
&
\PicS_{X_{\mbQ}}

\ar[leftarrow, sloped, near end]{dl}{}
& & \PicS_{Y_{\mbQ}} \ar[leftarrow, near start]{ll}{}
\ar[leftarrow]{dl}
\\
{\hphantom{aaa}\, \PicS_{X_{\mbQ}} \times_{\PicS_{Y_{\mbQ}}} \PicS_{Y'_{\mbQ}}}

& & \PicS_{Y'_{\mbQ}} \ar[leftarrow]{ll}
\ar[crossing over, near start, leftarrow]{uu}{}
\end{tikzcd}
\end{center}
The top and the bottom facets are $2$-pullback squares by definition and the right one is a $2$-pullback square up to inverting $p$ by Theorem \ref{thm:univ_homeo_and_Pic_intro}. By the same argument as in Step 3 of the above proof, the left facet is a $2$-pullback square up to inverting $p$, too. 

By Theorem \ref{thm:univ_homeo_and_Pic_intro} and the above paragraph, the big square and the right square in the following diagram
\begin{center}
\begin{tikzcd}
\PicS_{X'}  \arrow{d}  \arrow{r}  &  \PicS_{X} \times_{\PicS_Y} \PicS_{Y'} \arrow{d} \arrow{r} & \PicS_{X}  \arrow{d}  \\
\PicS_{X'_{\mbQ}}    \arrow{r} & \PicS_{X_\mbQ} \times_{\PicS_{Y_{\mbQ}}} \PicS_{Y'_{\mbQ}} \arrow{r} & \PicS_{X_{\mbQ}}, 
\end{tikzcd}
\end{center}
are $2$-pullbacks up to inverting $p$, hence so is the left one (\stacksproj{02XD}).
\end{proof}

\section{Pushouts of universal homeomorphisms in mixed characteristic} \label{s:pushouts_of_universal_homeomorphisms_in_mixed_char}
The goal of this section is to prove Theorem \ref{thm:general_pushouts_in_mixed_char_intro}. The following proposition is a key component of its proof.
\begin{proposition} \label{prop:extending_universal_homeo_algebras} Let $B$ be a ring and let $A' \to B'$ be a universal homeomorphism of $\mbQ$-algebras, where $B' \coloneq B \otimes_{\mbZ} \mbQ $. Then $A \to B$ is a universal homeomorphism where $A := B \times_{B'} A'$ is the pullback of the diagram
\begin{center}
\begin{tikzcd}
B \arrow{r}{i} & B' \\
A \arrow{u} \arrow{r} & A' \arrow{u}{j}.
\end{tikzcd}
\end{center}
Therefore, $\Spec A$ is the geometric pushout of $\Spec B \leftarrow \Spec B' \to \Spec A'$.
\end{proposition}
\begin{proof}
Note that $A \otimes_{\mbZ} \mbQ = A'$. {By Lemma \ref{lem:verify_uh_over_every_p}, it is enough to check that $A \to B$ is a universal homeomorphism after tensoring by $\mathbb{Z}_{(p)}$ for every prime $p$. Since tensoring by $\mathbb{Z}_{(p)}$  is equivalent to inverting all prime numbers $q \neq p$, we have that $A \otimes \mathbb{Z}_{(p)} \simeq (B\otimes \mathbb{Z}_{(p)}) \times_{B'} A'$. Hence, by replacing $B$ by $B \otimes \mathbb{Z}_{(p)}$, we may assume that $A$ and $B$ are defined over $\mathbb{Z}_{(p)}$.} In particular, $A' = A[\frac{1}{p}]$ and $B' = B[\frac{1}{p}]$.

First, we reduce to the case of $A' \to B'$ being a finite universal homeomorphism. By Proposition \ref{prop:0CND}, we can find $A'$-subalgebras $B'_{\lambda} \subseteq B'$ such that $B' = \varinjlim B'_{\lambda}$ and $A' \to B'_{\lambda}$ are finite universal homeomorphisms. For $B_{\lambda} \coloneq i^{-1}(B'_{\lambda}) \subseteq B$, we have $B_{\lambda}[\frac{1}{p}] = B'_{\lambda}$. Assume that 
\[
	A_{\lambda} \coloneq B_{\lambda} \times_{B'_{\lambda}} A' \to B_{\lambda}
\]
are universal homeomorphisms where $A_{\lambda} \subseteq A$. Then $A = \varinjlim A_{\lambda} \to B = \varinjlim B_{\lambda}$ has a locally nilpotent kernel and is a universal homeomorphism by Proposition \ref{prop:OCNE}.

As of now, we can assume that $A' \to B'$ is finite. Thus, by Proposition \ref{prop:elementary_extensions}, the morphism $A' \to B'$ can be factorised as 
\[
A' \to A'/I' =: B'_0 \hookrightarrow B'_1 \hookrightarrow \ldots \hookrightarrow B'_k := B',
\]
where $I'$ is a locally nilpotent ideal, and $B'_{i-1} \subseteq B'_i$ is an elementary extension for $1 \leq i \leq k$. It is enough to prove the proposition for each subsequent morphism separately, so we may assume that $A' \to B'$ is either a surjection with a locally nilpotent ideal, or $A' \hookrightarrow B'$ is an elementary extension. 

First, assume that $A' \to B'$ is surjective with locally nilpotent ideal $I'$. Then 
\[
A = B \times_{B'} A' \to B
\]
is also surjective with the kernel $0 \times_{B'} I' \subseteq A$ being locally nilpotent. Therefore, $A \to B$ is a universal homeomorphism.

Thus, we can assume that $A' \subseteq B'$ is an elementary extension, i.e.\ there exists $f' \in B'$ such that $A'[f'] = B'$ and $f'^2,f'^3 \in A'$. In particular, $A \to B$ is an inclusion. Since $A$ is constructed as a product, it is saturated inside $B$, that is, if $b \in B$ is such that $p^lb \in A$ for some $l>0$, then $b \in A$. Indeed, the image of $b$ in $B'$ is, by assumption, contained in the image of $A'$.

Now, by multiplying $f' \in B'$ by a power of $p$, we may assume that $f'$ is the image of an element $f \in B$ such that $f^2, f^3 \in A$. Therefore,  for every $b \in B$, we have $f^2b \in B$ and $p^lf^2b \in A$ for some $l>0$, hence $f^2b \in A$, and so $f^2B \subseteq A$. 

Consider $A/ (fB \cap A) \subseteq B/fB$. We claim that given $[b] \in B/fB$, we have $[b^{p^k}] \in A/(fB\cap A)$ for some $k>0$. In other words, for $b \in B$, there exists $k > 0$, $a \in A$, and $b' \in B$ such that $b^{p^k} = a + b'f$. 

Since $A' \subseteq B'$ is an elementary extension,
\[
p^kb = a_1 + a_2f
\]  
for some $k>0$ and $a_1,a_2 \in A$. Write
\begin{align*}
a_1^{p^k} &= (p^kb - a_2f)^{p^k}\\ 
		  &= p^{kp^k}b^{p^k} - {p^k \choose 1}p^{(p^k-1)k}b^{p^k-1}a_2f + f^2q\\
		  &= p^{kp^k}\underbrace{(b^{p^k} - b^{p^k-1}a_2f)}_{=:\, a} + f^2q,
\end{align*}
for some $q \in B$. Since $a_1^{p^k} \in A$ and $f^2q \in A$, we have that $p^{kp^k}a \in A$, and so $a \in A$. Write
\[
b^{p^k} = a + b^{p^k-1}a_2 f.
\] 
Thus, the claim holds for $b' := b^{p^k-1}a_2 \in B$.

Now, we will show that $b^{p^{k+l}} \in A[f] \subseteq B$ for some $l>0$. To this end, take $l>0$ such that $p^lb' = a'_1 + a'_2f$ for $a_1', a_2' \in A$ and write
\begin{align*}
b^{p^{k+l}} &= (a + b'f)^{p^l} \\
			&= a^{p^l} + p^la^{p^l-1}b'f  + f^2q' \\
			&= a^{p^l} + a^{p^l-1}(a'_1 + a'_2f)f + f^2q' \\
			&= \underbrace{(a^{p^l} + (a^{p^l-1}a'_2+q')f^2)}_{\in A} + \underbrace{a^{p^l-1}a'_1}_{\in A}f,
\end{align*}
where $q' \in B$.

By Proposition \ref{prop:0CND} and Lemma \ref{lem:univ_homeo_iso_generic}, respectively, $A \hookrightarrow A[f]$ and $A[f] \hookrightarrow B$ are universal homeomorphisms. Hence $A \hookrightarrow B$ is a universal homeomorphism. The last assertion of the proposition follows by Lemma \ref{lem:0et0}.
\end{proof}
\noindent Note that the morphism $A \to B$ need not be finite even when $A' \to B'$ is~so.

\begin{corollary} \label{cor:extending_uh}
Let $X$ be a scheme and let $X_{\mbQ} \to X'_{\mbQ}$ be a universal homeomorphism of schemes. Then a geometric pushout $X'$ of $X \leftarrow X_{\mbQ} \to X'_{\mbQ}$ exists as a scheme. The same statement holds for algebraic spaces if $X_{\mbQ} \to X'_{\mbQ}$  is representable. 
\end{corollary}
\begin{proof}
The morphism $X_{\mbQ} \to X$ is a base change of $\Spec \mbQ \to \Spec \mbZ$, hence it is affine, quasi-compact, and quasi-separated. By Lemma \ref{lem:construct_pushout_locally} it is enough to construct the pushout locally, hence we can assume that $X$ and $X_{\mbQ}$ are affine. Then $X'_{\mbQ}$ is affine by Remark \ref{rem:scheme_facts}(\ref{itm:05YU}). Now, the corollary is a consequence of Proposition \ref{prop:extending_universal_homeo_algebras}.
\end{proof}




\begin{example} \label{example:extending_fails_in_char_0} By \cite[Lemma 8.4]{kollar97}, Corollary \ref{cor:extending_uh} holds when $\mbZ$ and $\mbQ$ are replaced by $k[t]$ and $k(t)$ for a positive characteristic field $k$. However, when $k$ is of characteristic zero, Corollary \ref{cor:extending_uh} is false. To see this, take
\begin{align*}
R &:= \mbC[t],\\
B &:= \mbC[t][x,y],\\
B' &:= \mbC(t)[x, y], \text { and }\\
A' &:= \mbC(t)[x^2, x^3, x+ty].
\end{align*}
Assume that $\Spec B' \to \Spec A'$ extends to a universal homeomorphism $\Spec B \to \Spec C$. Then this morphism must factorise as $\Spec B \to \Spec A \to \Spec C$, where $A := B \times_{B'} A' \subseteq B$. Thus $\Spec B \to \Spec A$ is a universal homeomorphism as well, and, by Lemma \ref{lem:Eakin-Nagata}, $A$ is finitely generated over $R$. We shall show that this is not true. 

First, $x^n \in A'$ for all $n\geq 2$ as it is generated by $x^2$ and $x^3$. Moreover, $x^ny^k \in A'$ for all $n \geq 2$ and $k \geq 0$, by induction on $k$ and the formula
\[
x^n(x+ty)^k = x^{n}f(x,y) + t^kx^ny^k,
\]
where $f(x,y) \in B'$ and $\deg_y f(x,y) < k$. Thus
\[
x^2B' \subseteq A',
\]
is an ideal and $A' / x^2B' \simeq \mbC(t)[x+ty]$. Therefore, given $b' \in B'$, we have $b' \in A'$ if and only if 
\[
b' = x^2f(x,y) + a_0 + a_1(x+ty) + \ldots + a_m(mx(ty)^{m-1} + (ty)^m),
\]
for $m \in \mbN$, $f(x,y) \in B'$, and $a_k \in \mbC(t)$ where $0 \leq k \leq m$. 

This implies that given $b \in B$, we have $b \in A$ if and only if
\[
b = x^2f(x,y) + a_0 + a_1(x+ty) + \ldots + a_m(mxy^{m-1} + ty^m), 
\]
for $m \in \mbN$, $f(x,y) \in B$, and $a_k \in \mbC[t]$. In particular, we see that
\[
A / (xB \cap A) \simeq \mbC[t][ty, \ldots, ty^k, \ldots]
\]
which is not finitely generated over $R$, and so neither is $A$. The same holds true for $\mbC[t]_{(t)}$ instead of $\mbC[t]$.\\

This argument does not provide a counterexample to Corollary \ref{cor:extending_uh}, with $t$, from above, replaced by a prime number $p$, i.e.\ for 
\[
A' = \mbQ[x^2,x^3, x+py] \subseteq \mbQ[x,y] =: B', \text{ and } B = \mbZ_{(p)}[x,y].
\] 
What is different is that $xy^{p-1} + y^p \in A$ as
\[
(x+py)^p \equiv p^p(xy^{p-1} + y^p) \text{ mod } x^2.
\]
It is not difficult to see that, in this setting, $A$ is generated by $x^2y^i$ for $0 \leq i < p$, $x^3y^j$ for $0 \leq j < p$, $mxy^{m-1} + py^m$ for $1 \leq m < p$, and $xy^{p-1} + y^p$.
\end{example}

We are ready to give a proof of Theorem \ref{thm:general_pushouts_in_mixed_char_intro}. 

\begin{theorem}[cf.\ Theorem \ref{thm:general_pushouts_in_mixed_char_intro}] \label{thm:general_pushouts_in_mixed_char} Let $X \xleftarrow{p} Y \xrightarrow{g} Y'$ be a diagram of schemes or algebraic spaces such that $p$ is representable, quasi-compact, and \emph{separated}, and $g$ is a representable universal homeomorphism. Assume that a topological pushout of $X_{\mbQ} \leftarrow Y_{\mbQ} \rightarrow Y'_{\mbQ}$ exists as a scheme or an algebraic space, respectively. Then a geometric pushout of $X \leftarrow Y \to Y'$ exists as a scheme or an algebraic space, respectively.
\end{theorem}
\begin{proof} By Lemma \ref{lemma:technical_pushouts} a geometric pushout of $X_{\mbQ} \leftarrow Y_{\mbQ} \rightarrow Y'_{\mbQ}$ exists, and by Lemma \ref{lem:construct_pushout_locally} we can assume that $X$, $Y$, $Y'$ are schemes, while preserving the fact that a geometric pushout of $X_{\mbQ} \leftarrow Y_{\mbQ} \rightarrow Y'_{\mbQ}$ exists. By Remark \ref{rem:scheme_facts}(\ref{itm:07VV}) the geometric pushout $Z$ of $X_{\mbQ} \leftarrow Y_{\mbQ} \to Y'_{\mbQ}$ is then also a scheme. We split the proof into four steps.\\

\noindent \textbf{Step 1.} We reduce to the case when $Y_{\mbQ} \to Y'_{\mbQ}$ is an isomorphism.\\

Let $\tilde X$ be the geometric pushout of $Z \leftarrow X_{\mbQ} \to X$ which exists by Corollary \ref{cor:extending_uh}. Let $\tilde Y \subseteq \tilde X \times Y'$ be the image of the map $Y \to \tilde X \times Y'$ induced by $Y \to X \to \tilde X$ and $Y \to Y'$.
\begin{center}
\begin{tikzcd}
Y_{\mbQ} \arrow{r} \arrow{d} & X_{\mbQ} \arrow{r} \arrow{d} & X \arrow{d} & Y \arrow{l} \arrow{dr} \arrow{dl} \arrow{d} & \\
Y'_{\mbQ} \arrow{r}  & Z \arrow{r} \arrow[ul, phantom, "\ulcorner", very near start] & \tilde X \arrow[ul, phantom, "\ulcorner", very near start] & \tilde Y \arrow{l}{\tilde p} \arrow{r}{\tilde g} & Y'
\end{tikzcd}
\end{center}

By construction, $Y \to \tilde Y$ is surjective, and so both $Y \to \tilde Y$ and $\tilde Y \to Y'$ are universal homeomorphisms by Lemma \ref{lemma:factors_of_universal_homeomorphism}. Moreover, $Y_{\mbQ} \to \tilde X_{\mbQ}$ factorises through $Y'_{\mbQ} \to Z \simeq \tilde X_{\mbQ}$, thus 
\[
\tilde Y_{\mbQ} = \mathrm{im}(Y_{\mbQ} \to \tilde X_{\mbQ} \times Y'_{\mbQ}) \simeq  Y'_{\mbQ}.
\]
By Lemma \ref{lemma:technical_pushouts}, it is enough to construct a geometric pushout $\tilde X'$ of $\tilde X \leftarrow \tilde Y \to Y'$. Therefore, by replacing $X \leftarrow Y \to Y'$ by this diagram, we can assume that $Y_{\mbQ} \to Y'_{\mbQ}$ is an isomorphism. Note that $\tilde p$ is quasi-compact and separated by Remark \ref{rem:scheme_facts}(\ref{itm:03E4})(\ref{itm:09MQ}), and so the assumptions of Theorem \ref{thm:general_pushouts_in_mixed_char} are preserved.\\



\noindent \textbf{Step 2.} We reduce to the case of $p$ being affine or a contraction.\\

By Lemma \ref{lem:construct_pushout_locally} we can assume that $X$ is affine, while preserving the fact that $p$ is quasi-compact separated and $Y_{\mbQ} \to Y'_{\mbQ}$ is an isomorphism. Since $X$ is affine, both $X$ and $Y$ are quasi-compact and separated. Thus, by the Zariski-Nagata compactification and the Stein factorisation (cf.\ \cite[Theorem 1.1.3]{temkin11}, \stacksproj{03H2}), the separated morphism $Y \to X$ can be factored as $Y \to X_1 \to X_0 \to X$, where the first and the third map are affine, and the second one is a contraction. Then, using Step 3 and Step 4, we can construct geometric pushouts $X_1'$, $X_0'$, and $X'$ of $X_1 \leftarrow Y \to Y'$, $X_0 \leftarrow X_1 \to X_1'$, and $X \leftarrow X_0 \to X_0'$, respectively. Note that these pushouts are trivial over $\mbQ$. 
\begin{center}
\begin{tikzcd}
X \arrow[dashed]{d}  & X_0 \arrow{l} \arrow[dashed]{d} &  X_1 \arrow{l} \arrow[dashed]{d} & Y \arrow{l} \arrow{d} \\
X' & X_0' \arrow[dashed]{l} &  X_1' \arrow[dashed]{l} &  Y' \arrow[dashed]{l} 
\end{tikzcd}
\end{center}
By standard diagram chase, $X'$ is the geometric pushout of $X \leftarrow Y \to Y'$.\\

\noindent \textbf{Step 3.} We assume that $p$ is affine (cf.\ \cite[Lemma 8.9]{kollar97}).\\

By Lemma \ref{lem:construct_pushout_locally} we can assume that $X$ is affine, while preserving the fact that $p$ is affine and $Y_{\mbQ} \to Y'_{\mbQ}$ is an isomorphism. In particular, $Y$ and $Y'$ are affine as well (see Remark \ref{rem:scheme_facts}(\ref{itm:05YU})). Let $X' = \Spec A \times_{B} B'$, where $X \leftarrow Y \to Y'$ corresponds to $A \to B \leftarrow B'$. By Lemma \ref{lem:0et0}, the diagram
\begin{center}
\begin{tikzcd}
\mathllap{X = }\ \Spec A  \arrow{d} & \Spec B\ \mathrlap{ = Y} \arrow{l} \arrow{d} \\
\mathllap{X' = }\ \Spec A \times_B B' &  \Spec B'\ \mathrlap{ = Y'}  \arrow{l}
\end{tikzcd}
\end{center}
is a geometric pushout provided that $X \to X'$ is a universal homeomorphism. To show that this is the case, we can assume that $X$ and $X'$ are defined over $\Zpp$ by Lemma \ref{lem:verify_uh_over_every_p}. Then Lemma \ref{lem:univ_homeo_iso_generic} shows that $B^{\perf} \simeq B'^{\perf}$. Thus
\[
A^{\perf} \leftarrow (A\times_B B')^{\perf} =A^{\perf} \times_{B^{\perf}} B'^{\perf} \simeq A^{\perf}  
\]
is an isomorphism, and so by Lemma \ref{lem:univ_homeo_iso_generic} again, $X \to X'$ is a universal homeomorphism.\\

\noindent \textbf{Step 4.} We assume that $p$ is a contraction.\\

By Lemma \ref{lem:construct_pushout_locally} we can assume that $X$ is affine, while preserving the fact that $p$ is a contraction and $Y_{\mbQ} \to Y'_{\mbQ}$ is an isomorphism. In particular, $Y$ and $Y'$ are quasi-compact and quasi-separated (see Remark \ref{rem:scheme_facts}(\ref{itm:qc})(\ref{itm:09MQ})). Set $X' \coloneq \Spec H^0(Y', \mcO_{Y'})$.
Since $H^0(Y, \mcO_Y)= H^0(X, \mcO_X)$, we get a commutative diagram (cf.\ \stacksproj{01I1}):
\begin{center}
\begin{tikzcd}
X  \arrow{d}{f} & Y \arrow{l}{p} \arrow{d}{g}  \\
X'  & Y' \arrow{l}{q}.  
\end{tikzcd}
\end{center}
 {We have that 
\begin{align*}
X'_{\mbQ} &=  \Spec H^0(Y',\mcO_{Y'})\otimes \mbQ \simeq \Spec H^0(Y'_{\mbQ},\mcO_{Y'_{\mbQ}}) \\
&\simeq \Spec H^0(Y_{\mbQ}, \mcO_{Y_{\mbQ}}) \simeq \Spec H^0(X_{\mbQ}, \mcO_{X_{\mbQ}}) = X_{\mbQ},
\end{align*}
where $X_{\mbQ} \simeq X'_{\mbQ}$ is induced by $f|_{X_{\mbQ}}$ and the first isomorphism follows from quasi-compactness of $Y'$.}

To show that $f \colon X \to X'$ is a universal homeomorphism, we can assume that the spaces are defined over $\Zpp$ by Lemma \ref{lem:verify_uh_over_every_p}. Then we have $H^0(X, \mcO_X)^{\perf} = H^0(Y, \mcO_Y)^{\perf} = H^0(Y', \mcO_{Y'})^{\perf}$ (by Lemma \ref{lem:univ_homeo_iso_generic} and quasi-compactness and quasi-separatedness of $Y$ and $Y'$). Hence, by Lemma \ref{lem:univ_homeo_iso_generic}, $X \to X'$ is a universal homeomorphism, and the geometric pushout exists by Lemma \ref{lemma:technical_pushouts}. \qedhere

\end{proof}

Using Corollary \ref{cor:extending_uh}, we also show the following lemma (generalising \cite[Lemma 2.1]{keel99}) which is essential in the proof of Theorem \ref{thm:quotients_algebraic_groups_intro}.
\begin{lemma} \label{lem:make_coequiliser_uni_homeo} Let $X$ be a quasi-compact quasi-separated algebraic space and let 
\[
R \xrightarrow{f} E \xrightrightarrows[q]{p} X
\]
be maps of algebraic spaces such that $p$, $q$ are representable quasi-compact quasi-separated, $f$ is a representable universal homeomorphism,  and $p \circ f = q \circ f$. Assume that there exists a representable universal homeomorphism $X_{\mbQ} \to X'_{\mbQ}$ such that the two composite morphisms $E_{\mbQ} \rightrightarrows X_{\mbQ} \to X'_{\mbQ}$ are identical. Then there exists a representable universal homeomorphism $X \to X'$ such that the two composite morphisms $E \rightrightarrows X \to X'$ are identical.
\end{lemma}
The lemma also holds in the category of schemes in which case the assumption on the quasi-compactness and quasi-separatedness of $X$ is not necessary.
\begin{proof}
By replacing $X$ by the geometric pushout of $X \leftarrow X_{\mbQ} \to X'_{\mbQ}$, which exists by Corollary \ref{cor:extending_uh}, we can assume that $p|_{E_{\mbQ}} = q|_{E_{\mbQ}}$.

First, we deal with the case when the spaces in question are schemes. To this end, we reduce the lemma to when $f_{\mbQ} \colon R_{\mbQ} \to E_{\mbQ}$ is an isomorphism. Let $E'$ be the geometric pushout (and hence a categorical pushout) of $R \leftarrow R_{\mbQ} \to E_{\mbQ}$. Then the induced map $f' \colon E' \to E$ is a universal homeomorphism and an isomorphism over $\mbQ$ (see Lemma \ref{lemma:technical_pushouts}). Moreover, $p \circ f' = q \circ f'$  by the universal property of categorical pushouts as  $p \circ f = q \circ f$ and $p_{\mbQ} = q_{\mbQ}$. Thus, we can conclude the reduction process by replacing $R$ by $E'$.

Set $X'=X$ as topological spaces and endow $X'$ with a structure of a ringed space by setting
$
\mcO_{X'} \coloneq \ker(p^* - q^* \colon \mcO_{X} \xrightarrow{} p_* \mcO_{E}).
$
Note that since $R=E$ topologically, we have a natural identification $p_*\mcO_E = q_*\mcO_E$. We claim that $\mcO_{X'}(U) \to \mcO_X(U)$ is a universal homeomorphism for every affine open subset $U \subset X$. To show the claim we can assume that our spaces are defined over $\Zpp$ by Lemma \ref{lem:verify_uh_over_every_p}. Then, by Lemma \ref{lem:univ_homeo_iso_generic}, $\mcO^{\perf}_{E} = \mcO^{\perf}_{R}$, and so $\mcO^{\perf}_{X'} = \mcO^{\perf}_{X}$. The claim follows by Lemma \ref{lem:univ_homeo_iso_generic} again.

Now, we claim that $X'|_U = \Spec \mcO_{X'}(U)$, and so $X'$ is a scheme with the induced map $g \colon X \to X'$ being a universal homeomorphism. By the above paragraph, $U = X'|_U = \Spec \mcO_{X'}(U)$ topologically. Now, by quasi-coherence of $p_* \mcO_E$ and exactness of localisation, we get that 
\[
\mcO_{X'}(D(f)) = \mcO_{X'}(U)_{f'},
\] 
where $f' \in \mcO_{X'}(U)$, its image in $\mcO_{X}(U)$ is denoted by $f$, and $D(f) \subseteq U$ is the complement of the locus where $f=0$ (cf.\ \stacksproj{01Z8}). This concludes the proof of the claim. That $g \circ p = g \circ q$, follows by construction.\\

Now we show the lemma for algebraic spaces. Let $U \to X$ be a surjective \'etale morphism from an affine scheme $U$ (which exists by \stacksproj{03H6} as $X$ is quasi-compact), and let $R_U$, $E^p_U$, and $E^q_U$ be its pullbacks via $p \circ f$, $p$, and $q$, respectively. Since the pullbacks of $E^p_U$ and $E^q_U$ under $f$ are isomorphic to $R_U$, we have a natural isomorphism $E^p_U \simeq E^q_U =: E_U$ by Remark \ref{rem:scheme_facts}(\ref{itm:07VW}), and so two maps $p_U, q_U \colon E_U \rightrightarrows U$. Moreover, $(p_U)_{\mbQ} = (q_U)_{\mbQ}$. 

Therefore, by the above paragraph, we can construct a universal homeomorphism $g_U \colon U \to U'$ equalising $p_U$ and $q_U$, and such that $U_{\mbQ} \simeq U'_{\mbQ}$. Since $U$ is affine and $X$ is quasi-separated, the morphism $U \to X$ is representable quasi-compact and separated (Remark \ref{rem:scheme_facts}(\ref{itm:03KS})(\ref{itm:03KR})). By Theorem \ref{thm:general_pushouts_in_mixed_char}, we can construct a geometric pushout $X'$ of $X \leftarrow U \to U'$ sitting inside the following diagram:
\begin{center}
\begin{tikzcd}
R \arrow{r}{f}  & E \arrow[shift right = 2pt]{r}[swap]{q} \arrow[shift left = 2pt]{r}{p}  & X  \arrow[dashed]{r}{g} & X'\\
R_U \arrow{r}{f_U} \arrow{u}  & E_U \arrow[shift right = 2pt]{r}[swap]{q_U} \arrow[shift left = 2pt]{r}{p_U} \arrow{u}  & U \arrow{u}  \arrow[dashed]{r}{g_U} & U' \arrow{u}.
\end{tikzcd}
\end{center}
In particular, the two compositions $E_U \to E \rightrightarrows X \to X'$ are identical, and since $E_U \to E$ is faithfully flat (and thus $\mcO_E \to \mcO_{E_U}$ is injective by \stacksproj{08WP}), the two compositions $E \rightrightarrows X \to X'$ are identical, too.
\end{proof}

\section{Gluing of semiampleness}
In order to prove Theorem \ref{theorem:main_intro} and Corollary \ref{cor:bpf_intro} we need to understand semiampleness on non-irreducible schemes.
\subsection{Gluing} \label{ss:gluing}
The following propositions follow  by the strategy of Keel given our Theorem \ref{thm:general_pushouts_in_mixed_char} and Corollary \ref{cor:pic_for_general_pushouts}.
\begin{proposition}[{cf.\ \cite[Corollary 2.9]{keel99}}] \label{prop:gluing} Let $X$ be a reduced scheme projective over a Noetherian base scheme $S$ and such that $X = X_1 \cup X_2$ for two {reduced} closed subschemes $X_1$ and $X_2$. Let $L$ be a line bundle on $X$ such that $L|_{X_1}$, $L|_{X_2}$, and $L|_{X_{\mbQ}}$ are semiample (EWM, resp.). Let $g_2 \colon X_2 \to Z_2$ be a morphism associated to $L|_{X_2}$. Assume that $g_2|_{X_1\cap X_2}$ has geometrically connected fibres. Then $L$ is semiample (EWM, resp.).
\end{proposition}
\begin{proof}
We can assume that $S$ is affine. Let $X_{1,2} \coloneq X_1 \cap X_2$ be the scheme theoretic interesection (in particular, it need not be reduced). Let $g_1 \colon X_1 \to Z_1$, $g_2 \colon X_2 \to Z_2$, and $g_{1,2} \colon X_{1,2} \to V_{1,2}$ be the morphisms associated to $L|_{X_1}$, $L|_{X_2}$, and $L|_{X_{1,2}}$, respectively. Let 
\[
V_{1,2} \xrightarrow{f_i} V_i \hookrightarrow Z_i
\]
be factorisations through the images $V_i$ of $V_{1,2}$ in $Z_i$. 
Note that $f_i$ are proper (Remark \ref{rem:scheme_facts}(\ref{itm:04NX})) with finite fibres, hence they are finite (Remark \ref{rem:scheme_facts}(\ref{itm:0A4X})). 
Moreover, since $g_2|_{X_{1,2}}$ has geometrically connected fibres, we get that $f_2$ is a finite universal homeomorphism.

We claim that a topological pushout of $(V_1)_{\mbQ} \leftarrow (V_{1,2})_{\mbQ} \to (V_2)_{\mbQ}$ exists. Indeed, let $g_{\mbQ} \colon X_{\mbQ} \to Z'$ be a map associated to $L|_{X_{\mbQ}}$ and let $V' \subseteq Z'$ be the image of $(X_{1,2})_{\mbQ}$. By construction, we get maps $(V_1)_{\mbQ}, (V_2)_{\mbQ} \to V'$ such that $(V_1)_{\mbQ} \to V'$ is proper (Remark \ref{rem:scheme_facts}(\ref{itm:04NX})) and a bijection on points (as $g_2|_{(X_{1,2})_{\mbQ}}$ has geometrically connected fibres), hence it is a finite universal homeomorphism (Remark \ref{rem:scheme_facts}(\ref{itm:0A4X})). In particular, $V'$ is the sought-for topological pushout.

Therefore, by Theorem \ref{thm:general_pushouts_in_mixed_char}, there exists a geometric pushout $V$ of $V_1 \leftarrow V_{1,2} \to V_2$ sitting in the following diagram
\begin{center}
\begin{tikzcd}
X_1 \arrow{dd}{g_1} \arrow[hookleftarrow]{rr} && X_{1,2} \arrow{d}{g_{1,2}} \arrow[hookrightarrow]{rr} && X_2 \arrow{dd}{g_2} \\
&& \arrow{ld}[swap]{f_1} V_{1,2} \arrow{rd}{f_2} && \\
Z_1 \arrow[hookleftarrow]{r} & V_1 \arrow[dashed]{dr}[swap]{h_1} & & V_2 \arrow[dashed]{dl}{h_2} \arrow[hookrightarrow]{r} & Z_2 \\
& & V. & &
\end{tikzcd}
\end{center}
By definition $f_2^* \colon \mcO_{V_2} \to (f_2)_*\mcO_{V_{1,2}}$ is injective, and hence so is $h_1^* \colon \mcO_{V} \to (h_1)_* \mcO_{V_1}$ (Remark \ref{remark:geo_pushouts_affine_and_finite}(4)). Thus, by Lemma \ref{lem:Eakin-Nagata}, $V$ is proper over $S$. Since $V_1$ and $V_2$ are of finite type, the morphisms $h_1$ and $h_2$ are finite (cf.\ Remark \ref{remark:geo_pushouts_affine_and_finite}).

First, we consider the EWM case. To this end, let $Z'_1$, $Z'_2$, and $Z$ be the pushouts of $Z_1 \hookleftarrow V_1 \to V$, $V \leftarrow V_2 \hookrightarrow Z_2$, and $Z'_1 \hookleftarrow V \hookrightarrow Z'_2$ (equivalently $V \leftarrow V \sqcup V \hookrightarrow Z'_1 \sqcup Z'_2$), respectively, which exist and are of finite type over $S$  by Theorem \ref{thm:Kollar_pinching}. By Remark \ref{rem:scheme_facts}(\ref{itm:09MQ})(\ref{itm:03GN}), $Z'_1$, $Z'_2$, and $Z$ are proper. {Since $X$ is a categorical pushout of $X_1 \hookleftarrow X_{1,2} \hookrightarrow X_2$ (\stacksproj{0C4J}) and the constructed maps $X_1 \to Z$, $X_2 \to Z$ agree on $X_{1,2}$, we get an induced map $g \colon X \to Z$, which is proper (as it is a map between proper spaces).} {Note that $g \colon X \to Z$ is associated to $L$; indeed, its restriction to $X_i$ agrees with $g_i$ for $i \in \{1,2\}$ up to a finite map, and so a closed integral subscheme $V \subseteq X_i$ is contracted by $g$ if and only if it is contracted by $g_i$ if and only if $L|_V$ is not big.} 

Now, we move on to the semiample case of the proposition in which case $V$ is a scheme. Up to replacing $L$ by some power, the line bundles $L|_{X_1}$, $L|_{X_{1,2}}$, and $L|_{X_2}$ induce ample line bundles $A_{Z_1}$, $A_{Z_2}$, $A_{V_{1,2}}$ on $Z_1$, $Z_2$, and $V_{1,2}$, respectively. Let $A_{V_1} \coloneq A_{Z_1}|_{V_1}$ and let $A_{V_2} \coloneq A_{Z_2}|_{V_2}$. By construction, these line bundles induce an element $(A_{V_1}, A_{V_2}, \phi) \in \PicS_{V_1} \times_{\PicS_{V_{1,2}}} \PicS_{V_2}$ where $\phi$ is an isomorphism of their restriction to $V_{1,2}$. Now let $A_{V_{\mbQ}} \in \PicS_{V_{\mbQ}}$ be a line bundle on $V_{\mbQ}$ given as a pullback via $V_{\mbQ} \to V' \subseteq Z'$ of the line bundle induced by the semiample fibration $g_{\mbQ} \colon X_{\mbQ} \to Z'$ of $L|_{X_{\mbQ}}$. These constructions provide an isomorphism between the restrictions of $(A_{V_1}, A_{V_2}, \phi)$ and $A_{V_{\mbQ}}$ to $\PicS_{(V_1)_{\mbQ}} \times_{\PicS_{(V_{1,2})_{\mbQ}}} \PicS_{(V_2)_{\mbQ}}$. Therefore, Corollary \ref{cor:pic_for_general_pushouts} implies the existence of a compatible line bundle $A_V \in \PicS_V$ up to replacing $L$ by some power. In particular, there is a map  
\[
A_V \to (h_1)_* A_{V_1} \times_{(h_1 \circ f_1)_*A_{V_{1,2}}} (h_2)_* A_{V_2}, 
\] 
of quasi-coherent sheaves. In fact, this is an isomorphism as can be checked locally in which case this is equivalent to $V$ being a geometric pushout. Moreover, $A_V$ is ample by \stacksproj{0B5V}.


We get the following diagram
\begin{center}
\begin{tikzcd}
H^0(Z_1, A_{Z_1}) \arrow{r} \arrow{d} & H^0(V_{1,2}, A_{V_{1,2}})  \arrow{d}{=} & \arrow{l} \arrow{d} H^0(Z_2, A_{Z_2})\\
H^0(V_1, A_{V_1}) \arrow{r} & H^0(V_{1,2}, A_{V_{1,2}}) & \arrow{l} H^0(V_2, A_{V_2}).
\end{tikzcd}
\end{center}
where the vertical arrows, up to replacing $L$ by a multiple, are surjective by Serre vanishing. The fibre product of the bottom row is $H^0(V,A_V)$ (as proved in the above paragraph), and since $H^0(X_i, L|_{X_i}) = H^0(Z_i, A_{Z_i})$ and $H^0(X_{1,2}, L|_{X_{1,2}}) = H^0(V_{1,2}, A_{V_{1,2}})$ the fibre product of the upper row is $H^0(X,L)$ (cf.\ \stacksproj{0B7M}). Hence, we get a surjective map between the fibre products of both rows
\[
H^0(X,L) \to H^0(V, A_V),
\] 
and so the base locus of $L$ is disjoint from $X_{1,2}$. When lifting sections via $H^0(X_i, L|_{X_i}) = H^0(Z_i, A_{Z_i}) \to H^0(V_i, A_{V_i})$ we can assume that they do not vanish at any given point disjoint from $X_{1,2}$, and hence $L$ is semiample. \qedhere

\end{proof}

\begin{proposition}[{cf.\ \cite[Lemma 2.10]{keel99}}] \label{prop:gluing_normalisation} Let $X$ be a reduced scheme projective over an excellent base scheme $S$. Let $\pi \colon Y \to X$ be its normalisation with $C \subseteq X$ and $D \subseteq Y$ being the conductors. Let $L$ be a line bundle on $X$ such that $\pi^*L$, $L|_C$, and $L|_{X_{\mbQ}}$ are semiample (EWM, resp.), and let $g \colon Y \to Z$ be the morphism associated to $\pi^*L$. Assume that $g|_D$ has geometrically connected fibres. Then $L$ is semiample (EWM, resp.).
\end{proposition}
\begin{proof}
Note that $X$ is a categorical pushout of $Y \hookleftarrow D \to C$ (see \cite[Proposition 2.29]{ct17} and \stacksproj{0E25}). 

Let $g_{D} \colon D \to V$ and $g_C \colon C \to V_2$ be the morphisms associated to $\pi^*L|_{D}$ and $L|_{C}$, respectively. These morphisms lie in the following diagram
\begin{center}
\begin{tikzcd}
 &&   D \arrow{d}{g_D} \arrow{drr}{\pi|_D}  &  & \\
 Y \arrow{d}{g} \arrow[hookleftarrow]{rru} 			& &  V \arrow{dl}[swap]{f_1} \arrow{dr}{f_2} 			  & & C \arrow{d}{g_C}               \\{}
Z \arrow[hookleftarrow]{r} 			 &  V_1	&		          & V_2  \arrow["="]{r} & V_2,                    
\end{tikzcd}
\end{center}
\noindent where $V_1$ is the image of $D$ under $g$. Since $g|_{D}$ has geometrically connected fibres, we get that $f_1$ is a universal homeomorphism. 

Arguing as in the proof of the above proposition, we can construct pushouts $V'$ and $Z'$ of $V_1 \leftarrow V \to V_2$ and $Z \leftarrow V_1 \to V'$, respectively. We get an induced map $X \to Z'$ and it is associated to $L$ in the EWM case. {Indeed, $Y \to X \to Z'$ factorises into $Y \to Z$ and the finite map $Z \to Z'$; in particular, an integral subscheme $V \subseteq X$ is contracted by $X \to Z'$ if and only if a surjective-onto-$V$ integral component $V' \subseteq \pi^{-1}(V)$ is contracted by $Y \to Z$ if and only if $\pi^*L|_{V'}$ is not big if and only if $L|_V$ is not big (cf.\ Lemma \ref{lem:pullback-of-big}).}  

In the semiample case, we proceed mutatis mutandis as in the proof of the above proposition.
\end{proof}

For the proof of Corollary \ref{cor:bpf_intro}, we also need the following result. 
\begin{proposition}[{cf.\ \cite[Corollary 2.12 and 2.14]{keel99}}] \label{prop:gluing_residue_field_finite}
Proposition \ref{prop:gluing} and Proposition \ref{prop:gluing_normalisation} hold true for when $g_2|_{X_1\cap X_2}$ and $g|_D$, respectively, have all geometric fibres, except for a finite number over closed points, being connected, provided we assume in the semiample case that {positive characteristic closed points have locally finite residue fields.}
\end{proposition}
\begin{proof}
We focus on the case of Proposition \ref{prop:gluing} as the case of Proposition \ref{prop:gluing_normalisation} is analogous. 
Let $T \subseteq V_2$ be the finite set of closed points over which the fibres of $g_2|_{X_{1,2}}$ are not connected and set $G \coloneq g_2^{-1}(T)$. We would like to apply Proposition \ref{prop:gluing} to $(X_1 \cup G) \cup X_2$. To this end, we need to verify that $L|_{X_1 \cup G}$ is semiample (EWM, resp.). 

Let $g_1 \colon X_1 \to Z_1$ be the morphism associated to $L|_{X_1}$. Since $g_1(X_1 \cap G)$ is a finite number of {closed} points, we have that $L|_{G'}$ is numerically trivial where $G' \coloneq g_1^{-1}(g_1(X_1 \cap G))$. Now, we apply Proposition \ref{prop:gluing} again to $X_1 \cup (G \cup G')$, wherein $L|_{G \cup G'}$ is numerically trivial, and hence semiample (EWM, resp.) as each connected component of $G \cup G'$ is of finite type over {a locally finite field or a field of characteristic zero in the semiample case}, cf.\ \cite[Lemma 2.16]{keel99}. This concludes the proof.
\end{proof}
\section{Proofs of the main theorems}

\subsection{Keel's base point free theorem in mixed characteristic}
As pointed out in the introduction, the key to the proof of Theorem \ref{theorem:main_intro} is Theorem \ref{thm:reduced_semiample}. 

In what follows, we consider a category of pairs $(X,L_X)$ consisting of a scheme with a line bundle $L_X$ on it, and we denote by $f \colon (X,L_X) \to (Y,L_Y)$ a data of a morphism $f \colon X \to Y$ together with an isomorphism $f^*L_Y \simeq L_X$. 
\begin{proof}[Proof of Theorem \ref{thm:reduced_semiample}]
We start with the EWM case of the theorem. Let $g \colon X^{\red} \to Z$ be a map associated to $L|_{X^{\red}}$. We claim that there exists a topological pushout $Z'$ of $X \leftarrow X^{\red} \to Z$ which is proper over $S$. To this end, let $X_{\mbQ} \to Z'_{\mbQ}$ be a contraction associated to $L|_{X_{\mbQ}}$. The induced map $Z_{\mbQ} \to Z'_{\mbQ}$ is proper (Remark \ref{rem:scheme_facts}(\ref{itm:04NX})) and a bijection on geometric points, hence a finite universal homeomorphism (Remark \ref{rem:scheme_facts}(\ref{itm:0A4X})). Thus, $Z'_{\mbQ}$ is a topological pushout of $X_{\mbQ} \leftarrow X^{\red}_{\mbQ} \to Z_{\mbQ}$, and hence the claim follows by Theorem \ref{thm:general_pushouts_in_mixed_char} and Lemma \ref{lem:fin_gen_pushouts}. Now, the induced map $X \to Z'$ is one associated to $L$ {(the condition of being a map associated to a line bundle depends on the reduction only)}.



We move on to the semiample case. We can assume that $S$ is an affine Noetherian scheme over $\mathbb{Z}_{(p)}$ where $p$ is a prime number. The semiample line bundles $L|_{X^{\red}}$, $L|_{X^{\red}_{\mbQ}}$, and $L|_{X_{\mbQ}}$, up to replacing $L$ by some power, induce the following commutative diagram {(in the category of pairs as stated above which enforces compatiblity of line bundles and their isomorphisms)}

\begin{center}
\begin{tikzcd}
(X,L) & (X_{\mbQ}, L|_{X_{\mbQ}}) \arrow{l} \arrow[bend left = 60]{ddd} \\
(X^{\red}, L|_{X^{\red}}) \arrow{d} \arrow{u} & (X^{\red}_{\mbQ}, L|_{X^{\red}_{\mbQ}}) \arrow{u} \arrow{d} \arrow{l} \\
(Z, A) & (Z_{\mbQ}, A_{\mbQ}) \arrow{d} \arrow{l} \\
  & (Z'_{\mbQ}, A'_{\mbQ}), 
\end{tikzcd}
\end{center}
where $A$, $A_{\mbQ}$, and $A'_{\mbQ}$ are ample line bundles. Furthermore, since $X^{\red}_{\mbQ} \to X_{\mbQ}$ is a universal homeomorphism, so is $Z_{\mbQ} \to Z'_{\mbQ}$; { indeed, $|X^{\red}_{\mbQ}| =|X_{\mbQ}|$ and $X_{\mbQ} \to Z'_{\mbQ}$ has geometrically connected fibres, and so $X^{\red}_{\mbQ} \to Z'_{\mbQ}$ and the finite part $Z_{\mbQ} \to Z'_{\mbQ}$ of its Stein factorisation must have geometrically connected fibres as well} (however, $Z_{\mbQ} \to Z'_{\mbQ}$ need not necessary be a thickening). By Corollary \ref{cor:extending_uh} and Lemma \ref{lem:fin_gen_pushouts} we can construct a topological pushout $Z'$ of $Z \leftarrow Z_{\mbQ} \to Z'_{\mbQ}$ which is proper over $S$. Since $Z$ is of finite type, the induced map $Z \to Z'$ is a finite universal homeomorphism.  Thus, by Theorem \ref{thm:univ_homeo_and_Pic}, up to replacing $L$ by some power, we can extend the bottom left corner of the above diagram to a commutative square
\begin{center}
\begin{tikzcd}
(Z, A) \arrow{d} & (Z_{\mbQ}, A_{\mbQ}) \arrow{d} \arrow{l} \\
(Z', A') & (Z'_{\mbQ}, A'_{\mbQ}) \arrow{l}, 
\end{tikzcd}
\end{center}
such that $A'$ is ample (see \stacksproj{0B5V}). Applying $(H^0)^{\perf}$, we get a diagram
\begin{center}
\begin{tikzcd}
H^0(X,L)^{\perf} \arrow{r} \arrow{d} & H^0(X_{\mbQ}, L|_{X_{\mbQ}})^{\perf} \arrow{d}  \\
H^0(X^{\red},L|_{X^{\red}})^{\perf} \arrow{r} & H^0(X^{\red}_{\mbQ}, L|_{X^{\red}_{\mbQ}})^{\perf} \\
H^0(Z',A')^{\perf} \arrow[dashed, bend left = 90]{uu} \arrow{u} \arrow{r} & H^0(Z'_{\mbQ}, A'_{\mbQ})^{\perf}, \arrow{u} \arrow[bend right = 90]{uu}
\end{tikzcd}
\end{center}
where the left bent arrow exists by the Cartesianity of the upper square (see Proposition \ref{prop:sections_on_thickenings}) and the fact that $H^0(Z',A')^{\perf}$ maps compatibly to all other spaces in the above diagram. Since $A'$ is ample and $Z'$ is of finite type over $S$, we get that $A'$ is semiample (cf.\ \stacksproj{01VS}), and thus so is $L$.
\end{proof}
\noindent One could also tackle the semiample case of Theorem \ref{thm:reduced_semiample} by Theorem \ref{thm:general_pushouts_in_mixed_char} and Corollary \ref{cor:pic_for_general_pushouts}, but we believe that the above proof shows better what is really happening.  
\begin{theorem}[Theorem \ref{theorem:main_intro}] \label{theorem:main}  Let $L$ be a nef line bundle on a scheme $X$ projective over an excellent base scheme $S$. Then $L$ is semiample over $S$ if and only if both $L|_{\mathbb{E}(L)}$ and $L|_{X_{\mbQ}}$ are so. If $S$ is of finite type over a mixed characteristic Dedekind domain, then $L$ is EWM if and only if $L|_{\mathbb{E}(L)}$ and $L|_{X_{\mbQ}}$ are EWM.
\end{theorem}
\begin{proof}
We can assume that $S$ is affine. We proceed by Noetherian induction on $X$ as in \cite{keel99}. By Theorem \ref{thm:reduced_semiample}, we can assume that $X$ is reduced.


First, we reduce to the case of $X$ being irreducible. If $\mathbb{E}(L)=X$, then we are done, so may assume that there exists an irreducible component $X_1 \subseteq X$ such that $L|_{X_1}$ is big. Let $X_2 \subseteq X$ be the union of all the other irreducible components. Write
\[
X = X_1 \cup (X_2 \cup \mathbb{E}(L)).
\]
Assume that $L|_{X_1}$ is semiample (EWM, resp.) and let $g_1 \colon X_1 \to Z_1$ be an associated morphism. The exceptional locus of $g_1$ is contained in $\mathbb{E}(L)$, and hence $g_1$ has geometrically connected fibres on $X_1 \cap (X_2 \cup \mathbb{E}(L))$. Thus $L$ is semiample (EWM, resp.) if $L|_{X_2 \cup \mathbb{E}(L)}$ is semiample (EWM, resp.) by Proposition \ref{prop:gluing}. Repeating this process for $X_2 \cup \mathbb{E}(L)$ we see that it is enough to show the theorem for $X$ being irreducible. In particular, we can assume that $S$ is integral.

Since $L$ is big, we have that {$L^r \simeq A \otimes E$} for some $r \in \mbN$ where $A$ is an ample line bundle and {$E$ is a line bundle for which $H^0(X, E) \neq 0$. Let $D$ be a zero locus of some section $0 \neq s \in H^0(X,E)$.} {By definition, $\mathbb{E}(L|_{mD}) \subseteq \mathbb{E}(L)$; hence $L|_{\mathbb{E}(L|_{mD})}$ is semiample (EWM, resp.), and so by Noetherian induction,} $L|_{mD}$ is semiample (EWM, resp.) for every $m \in \mbN$ as well. 

In the semiample case, pick $k \gg m \gg 0$ divisible enough so that $L^k|_{mD}$ is base point free and $A^k$ is very ample. Consider the following exact sequence{
\[
H^0(X, L^k) \to H^0(mD, L^k|_{mD}) \to H^1(X, L^k(-mD))=0,
\]}
wherein the last cohomology group is zero by the Fujita vanishing (\cite[Theorem 1.5]{keeler03} and \cite{keeler03errata}) as {$L^k(-mD) \simeq A^m \otimes L^{k-mr}$}. Thus $L^k$ has no base points along $D$ and hence is base point free as {$L^k \simeq A^{k/r} \otimes E^{k/r}$ and $A^{k/r}$} is very ample.

The EWM case follows from \cite[Theorem 3.1 and Theorem 6.2]{Artin70} as in \cite[Proposition 1.6]{keel99}. Here, we need to assume that $S$ is of finite type over an excellent Dedekind domain to apply \cite{Artin70}.
\end{proof}   

\subsection{Quotients by finite equivalence relations in mixed characteristic}
As in Subsection \ref{ss:finite_quotients}, all geometric quotients are assumed to be separated and of finite type over a Noetherian base scheme $S$. The following lemma allows for constructing quotients of non-reduced schemes.
\begin{lemma} \label{lem:finite_quotients_after_reduction}
Let $X$ be a separated algebraic space of finite type over a Noetherian base scheme $S$. Let $E \rightrightarrows X$ be a finite, set theoretical equivalence relation and assume that the quotients $X_{\mbQ} / E_{\mbQ}$ and $X_{\red}/E_{\red}$ exist as separated algebraic spaces of finite type over $S$, where $X_{\red}$ and $E_{\red}$ are reductions of $X$ and $E$, respectively. Then the geometric quotient $X / E$ exists as a separated algebraic space of finite type over $S$.
\end{lemma}
\begin{proof}
Consider the following commutative diagram
\begin{center}
\begin{tikzcd}
E_{\red} \arrow{d} \arrow[shift right = 2pt]{r} \arrow[shift left = 2pt]{r} & X_{\red} \arrow{d} \arrow{r} & X_{\red}/E_{\red} \arrow[dashed]{d}\\
E \arrow[shift right = 2pt]{r} \arrow[shift left = 2pt]{r} & X \arrow[dashed]{r} & Y,
\end{tikzcd}
\end{center}
where $Y$ is the geometric pushout of $X \leftarrow X_{\red} \to X_{\red}/E_{\red}$. Such a pushout exists by Theorem \ref{thm:general_pushouts_in_mixed_char} as $X_{\mbQ} \leftarrow X_{\mbQ,\red} \to X_{\mbQ,\red}/E_{\mbQ, \red}$ admits a topological pushout in the form of $X_{\mbQ}/E_{\mbQ}$. Here, the map 
\[
X_{\mbQ,\red}/E_{\mbQ, \red} \to X_{\mbQ}/E_{\mbQ}
\]
is proper (Remark \ref{rem:scheme_facts}(\ref{itm:08AJ})), and a bijection on geometric points, hence a finite universal homeomorphism (Remark \ref{rem:scheme_facts}(\ref{itm:0A4X})). Moreover the map $X \to Y$ is integral (Remark \ref{remark:geo_pushouts_affine_and_finite}(2)) and the map  $Y_{\mbQ} \to X_{\mbQ}/E_{\mbQ}$ is a representable universal homeomorphism (Lemma \ref{lemma:technical_pushouts}). 

We know that two compositions $E_{\red} \to E \rightrightarrows Y$ coincide. Moreover, the compositions $E_{\mbQ} \rightrightarrows Y_{\mbQ} \to X_{\mbQ}/E_{\mbQ}$ coincide as well and $Y$ is quasi-compact quasi-separated (Remark \ref{remark:geo_pushouts_affine_and_finite}(3)), so by Lemma \ref{lem:make_coequiliser_uni_homeo} there exists a representable universal homeomorphism $Y \to Y'$ such that the compositions $E \rightrightarrows X \to Y \to Y'$ coincide. Thus a geometric quotient $X/E$ exists by Theorem \ref{thm:kollar_quotients_exist}.
\end{proof}

Our proof of Theorem \ref{thm:quotients_intro} follows closely the strategy of Koll\'ar from \cite[Section 4]{kollar12} with the new component being Theorem \ref{thm:general_pushouts_in_mixed_char}.
\begin{proof}[Proof of Theorem \ref{thm:quotients_intro}]
We prove the theorem by induction on dimension. Set $d = \dim X$. By Lemma \ref{lem:finite_quotients_after_reduction}, we can assume that $X$ and $E$ are reduced.

First we show the theorem under the assumption that $X$ is normal. To this end, we set $E^{d} \subseteq E$ and $X^{d}\subseteq X$ to be the unions of $d$-dimensional irreducible components of $E$ and $X$, respectively. Write $X = X^d \sqcup X^{<d}$, where $X^{<d}$ is the union of all the other components of $X$. By \cite[Lemma 28]{kollar12}, $E^d \rightrightarrows X^d$ is a set theoretic finite equivalence relation and the geometric quotient $X^d/E^d$ exists by \cite[Lemma 21]{kollar12}. Define $X/E^d := X^d/E^d \sqcup X^{<d}$.

Let $Z \subseteq X$ be a reduced closed subscheme of dimension lower than $d$ such that $Z$ is closed under $E$ and the equivalences $E|_{X\backslash Z}$ and $E^d|_{X\backslash Z}$ coincide. For example, set {$Z = \sigma_2(\sigma_1^{-1}(X^{<d} \cup \sigma_2(E^{<d})))$, where $\sigma_1,\sigma_2 \colon E \rightrightarrows X$ gives the equivalence relation and $E = E^d \cup E^{<d}$ for $E^{<d}$ being the union of irreducible components of $E$ of dimension at most $d-1$. Since $\sigma_2(\sigma_1^{-1}(T))$ is stable under $E$ for any subset $T$ of $X$ by transitivity of equivalence relations, $Z$ is stable under $E$.}  Consider the following diagram
\begin{center}
\begin{tikzcd}
Z \arrow[hookrightarrow]{r} \arrow{d} &  X \arrow{d} \\
Z' \arrow[hookrightarrow]{r}    	   & X/E^d,
\end{tikzcd} 
\end{center}
where $Z'$ is the image of $Z$ in $X/E^d$, and $Z \to Z'$ is finite. We have that $E|_Z \rightrightarrows Z$ is a finite set theoretic equivalence relation on $Z$ and since the geometric fibres of $Z \to Z'$ are subsets of $E$-equivalence classes, we get an induced equivalence relation $E_{Z'} \rightrightarrows Z'$ (see \cite[Definition 26]{kollar12}).

Since $X_{\mbQ}/E_{\mbQ}$ exists, \cite[Corollary 18]{kollar12} implies that $Z'_{\mbQ}/E_{Z',\mbQ}$ exists too, and, by induction, so does the geometric quotient $Z'/E_{Z'}$. By Theorem \ref{thm:Kollar_pinching}, there exists a pushout
\begin{center}
\begin{tikzcd}
Z' \arrow{r} \arrow{d} &  X / E^d \arrow{d} \\
Z'/ E_{Z'} \arrow{r}					   & Y,
\end{tikzcd} 
\end{center}
with $X/E^d \to Y$ finite and $Y$ being an algebraic space of finite type over $S$ (and separated by Remark \ref{rem:scheme_facts}(\ref{itm:05Z2})). Since $X/E^d \to Y$ equalises $E \rightrightarrows X/E^d$, Theorem \ref{thm:kollar_quotients_exist} implies the existence of $X/E$, which, in fact, coincides with $Y$ as being a categorical pushout of the above diagram is equivalent to being a categorical quotient of $X/E^d$ by $E$ (see \cite[Proposition 25]{kollar12}). 

We move to the case when $X$ is not necessarily normal. Let $g \colon \tilde X \to X$ be its normalisation (which is finite by \stacksproj{0BB5}), let $\tilde E$ be the pullback of $E$ (see \cite[Definition 26]{kollar12}), and let $q \colon \tilde X \to \tilde X / \tilde E$ be the geometric quotient which exists by the above paragraphs. Set $X^*$ to be the image of $\tilde X$ under the diagonal map $(q,g) \colon \tilde X \to (\tilde X/\tilde E) \times_S X$.

Since $\tilde X$ is separated (as so is $X$), the diagonal map $\tilde X \to \tilde X \times_S \tilde X$ is a closed immersion and $(q,g)$ is finite. Thus $\tilde X \to X^*$ is proper (Remark \ref{rem:scheme_facts}(\ref{itm:04NX})), and so $X^* \to \tilde X / \tilde E$ and $X^* \to X$ are proper as well (Remark \ref{rem:scheme_facts}(\ref{itm:08AJ})(\ref{itm:05Z2})). Since the fibres of $\tilde X \to X$ are contained in the equivalence classes of $\tilde E$, the map $X^* \to X$ is a bijection on geometric points, and so a finite universal homeomorphism (Remark \ref{rem:scheme_facts}(\ref{itm:0A4X})). 

The diagram $\tilde X_{\mbQ} / \tilde E_{\mbQ} \leftarrow X^*_{\mbQ} \rightarrow X_{\mbQ}$ admits a topological pushout in the form of $X_{\mbQ}/E_{\mbQ}$. Indeed, the composite map $\tilde X_{\mbQ} \to \tilde X_{\mbQ} / \tilde E_{\mbQ} \to X_{\mbQ}/E_{\mbQ}$ is finite, and so $\tilde X_{\mbQ} / \tilde E_{\mbQ} \to X_{\mbQ}/E_{\mbQ}$ 
is proper (Remark \ref{rem:scheme_facts}(\ref{itm:08AJ})); as it is also a bijection on geometric points, it must be a finite universal homeomorphism (Remark \ref{rem:scheme_facts}(\ref{itm:0A4X})). Thus Theorem \ref{thm:general_pushouts_in_mixed_char} implies that the geometric pushout, say $W$, of $\tilde X / \tilde E \leftarrow X^* \rightarrow X$ exists.
\begin{center}
\begin{tikzcd}
\tilde X \arrow{dr} \arrow[bend right = 5]{ddr} \arrow[bend left = 5]{drr}  & & \\
& X^* \arrow{d} \arrow{r} & X \arrow[dashed]{d} \\
& \tilde X / \tilde E \arrow[dashed]{r} &  W.
\end{tikzcd}
\end{center} 

Moreover, $E$ is a set theoretic equivalence relation over $W$ (as $\mcO_{E} \to \mcO_{\tilde E}$ is injective due to $X$ and $E$ being reduced), and so the geometric quotient $X/E$ exists by Theorem \ref{thm:kollar_quotients_exist}. Note that $X \to W$ is integral by Remark \ref{remark:geo_pushouts_affine_and_finite}.
\end{proof}

\subsection{Quotients by affine algebraic groups in mixed characteristic}
Now, we move on to the proof of Theorem \ref{thm:quotients_algebraic_groups_intro}. To this end, we need the following lemma.
\begin{lemma} \label{lemma:quotients_group_schemes_universal_homeomorphism} Let $G$ be an affine algebraic group scheme, flat and of finite type over a Noetherian base scheme $S$, let $X$ and $Y$ be separated algebraic spaces of finite type over $S$ admitting a proper action of $G$, and let $f \colon X \to Y$ be a finite and universal $G$-homeomorphism. If the geometric quotient $Y/G$ exists, then so does $X/G$. Conversely, if both $X/G$ and $Y_{\mbQ}/G_{\mbQ}$ exist, then so does $Y/G$.
\end{lemma}
\begin{proof}
If the geometric quotient $Y/G$ exists, then $X/G$ exists by applying Theorem \ref{thm:kollar_quotients_group_schemes_exist} to $X \to Y \to Y/G$. 

Thus, we can assume that $X/G$ and $Y_{\mbQ}/G_{\mbQ}$ exist. Since $Y_{\mbQ}/G_{\mbQ}$ is a topological pushout of $X_{\mbQ}/G_{\mbQ} \leftarrow X_{\mbQ} \rightarrow Y_{\mbQ}$ (see Remark \ref{remark:group_action_facts}(5)), a geometric pushout $Z$ of $X/G \leftarrow X \to Y$ exists by Theorem \ref{thm:general_pushouts_in_mixed_char}.

We claim that there exists a representable universal homeomorphism $Z \to Z'$ such that the composite map $Y \to Z \to Z'$ is a $G$-morphism with $Z$ endowed with a trivial $G$-action. To this end, we consider the following commutative diagram:
\begin{center}
\begin{tikzcd}
X/G \times G \arrow{d}{m_{X/G}} & \arrow{l} X \times G \arrow{d}{m_X} \arrow{r} & Y \times G \arrow{d}{m_Y} \\
X/G & \arrow{l} X \arrow{r} & Y,
\end{tikzcd}
\end{center} 
where the vertical arrows are given by $G$-actions. In particular, we get an induced map $m_Z \colon Z \times G$ between the pushouts of both rows, such that the following diagram
\begin{center}
\begin{tikzcd}{}
X/G \times G \arrow{d}{m_{X/G}} \arrow{r} & Z \times G \arrow{d}{m_Z} \\
X/G \arrow{r} & Z
\end{tikzcd}
\end{center}
is commutative. Since $m_{X/G}$ is a projection, the two composite maps 
\[
X/G \times G \to Z \times G \xrightrightarrows[m_Z]{\pi} Z
\]
are identical, where $\pi$ is a projection. 

By Lemma \ref{lemma:technical_pushouts}, there exists a representable universal homeomorphism $Z_{\mbQ} \to Y_{\mbQ}/G_{\mbQ}$. Further, the two composite maps
\[
Z_{\mbQ} \times G_{\mbQ} \xrightrightarrows[m_Z]{\pi} Z_{\mbQ} \to Y_{\mbQ}/G_{\mbQ}
\]
are identical (here we used that $m_{Z_{\mbQ}} \colon Z_{\mbQ} \times G_{\mbQ} \to Z_{\mbQ}$ is compatible with $m_{Y_{\mbQ}/G_{\mbQ}} \colon Y_{\mbQ} / G_{\mbQ} \times G_{\mbQ} \to Y_{\mbQ}/ G_{\mbQ}$). Hence, we can invoke Lemma \ref{lem:make_coequiliser_uni_homeo} to get a representable universal homeomorphism $Z \to Z'$ such that the two composite maps $Z \times G \rightrightarrows Z \to Z'$ are identical. This concludes the proof of the claim. 

Given the claim, the geometric quotient $Y/G$ exists by Theorem \ref{thm:kollar_quotients_group_schemes_exist} applied to $Y \to Z'$.
\end{proof}

Note that a normalisation of an excellent scheme is finite (\stacksproj{07QV and 035S}).
\begin{proof}[Proof of Theorem \ref{thm:quotients_algebraic_groups_intro}]
Note that $X_{\mbQ}/G_{\mbQ}$ exists by the characteristic zero case of the theorem (see \cite[Theorem 9.16]{Viehweg95} for when $S_{\mbQ}$ is of finite type over a field). 

We follow the strategy described in \cite[5.7]{kollar97}. By \cite[Theorem 5.6]{kollar97}, the action of $G$ on $X$ lifts to the seminormalisation $X^{\mathrm{sn}}$ of the reduction of $X$. By Lemma \ref{lemma:quotients_group_schemes_universal_homeomorphism}, it is enough to show that a geometric quotient $X^{\mathrm{sn}}/G$ exists. Hence we can assume that $X$ is seminormal and reduced.

By \cite[Proposition 4.1]{kollar97}, the action of $G$ on $X$ lifts to the normalisation $X^n$ of $X$ and, by \cite[Theorem 4.3]{kollar97}, the geometric quotient $X^n/G$ exists. Let $C \subseteq X$ and $D \subseteq X^n$ be the conductor schemes. We must have that $C$ is $G$-invariant, and so $D$ admits a proper action of $G$. Moreover, it admits a topological quotient $D \to D_{X^n/G}$, where $D_{X^n/G}$ is the image of $D$ in $X^n/G$. Hence the geometric quotient $D/G$ exists by Theorem \ref{thm:kollar_quotients_group_schemes_exist} and the induced map $D/G \to D_{X^n/G}$ is a finite universal homeomorphism. We can assume that the geometric quotient $C/G$ exists by Noetherian induction. The induced map $D/G \to C/G$ is finite by Remark \ref{remark:group_action_facts}.

In \cite[Theorem 5.8]{kollar97}, it is shown that the geometric quotient $X/G$ exists provided that the geometric quotients $X^n/G$, $C/G$, and a topological pushout of $C/G \leftarrow D/G \to D_{X^n/G}$ (with the maps from $C/G$ and $D_{X^n/G}$ to the pushout being finite) exist.  Note that the image  $C_{X_{\mbQ}/G_{\mbQ}}$ of $C_{\mbQ}$ in $X_{\mbQ}/G_{\mbQ}$ is a topological pushout of $C_{\mbQ}/G_{\mbQ} \leftarrow D_{\mbQ}/G_{\mbQ} \to (D_{X^n/G})_{\mbQ}$, as $C_{\mbQ}/G_{\mbQ} \to C_{X_{\mbQ}/G_{\mbQ}}$ is a finite universal homeomorphism (Theorem \ref{thm:kollar_quotients_group_schemes_exist}). Hence, we can invoke Theorem \ref{thm:general_pushouts_in_mixed_char} to get a geometric pushout and Lemma \ref{lem:fin_gen_pushouts} to get a topological pushout $Z$ of $C/G \leftarrow D/G \to D_{X^n/G}$ of finite type over $S$. Then $C/G \to Z$ and $D_{X^n/G} \to Z$ are integral (Remark \ref{remark:geo_pushouts_affine_and_finite}) and hence finite as $C/G$ and $D_{X^n/G}$  are of finite type over $S$.  \qedhere
\end{proof}

\begin{remark} \label{rem:flat_group_quotients}
Let $h \colon G' \to G$ be a universal homeomorphism of flat group schemes of finite type over a Noetherian base scheme $S$. By the same argument as in Lemma \ref{lemma:quotients_group_schemes_universal_homeomorphism}, one can show that a geometric quotient by a proper action of $G$ exists if and only if a geometric quotient by a proper action of $G'$ exists, provided that both quotients exist over $\mbQ$. This allows for weakening the assumptions of Theorem \ref{thm:quotients_algebraic_groups_intro}.
\end{remark}

\subsection{Base point free theorem in mixed characteristic}
Throughout this subsection, we assume that $S$ is a quasi-projective scheme defined over a {mixed characteristic} Dedekind domain. Theorem \ref{theorem:main} immediately implies the existence of plt contractions for mixed characteristic threefolds.
\begin{corollary} \label{cor:bpf_plt} let $(X,D+B)$ be a plt pair on a normal integral mixed characteristic scheme $X$ of (absolute) dimension three which is projective over $S$, with $D$ being a normal irreducible divisor {and $B$ being an effective $\mathbb{Q}$-divisor}. Let $L$ be a nef Cartier divisor on $X$ such that $L-(K_X+D+B)$ is ample and $\mathbb{E}(L) \subseteq D$. Then $L$ is semiample.
\end{corollary}
\begin{proof}
By adjunction, \cite[Theorem 4.2]{tanaka16_excellent}, and \cite[Theorem 1.1]{tanakaimperfect}, $L|_D$ is semiample, and so $L|_{\mathbb{E}(L)}$ is semiample as well. Moreover, $L|_{X_{\mbQ}}$ is semiample by the base point free theorem in characteristic zero (cf.\ \cite[Theorem 3.9.1]{bchm06} or \cite[Theorem 5.1]{HK10}). Hence, $L$ is semiample by Theorem \ref{theorem:main}. 
\end{proof}

We move on to the proof of Corollary \ref{cor:bpf_intro}. To this end, we need the following result. {Here, $X$ can be of positive characteristic.}
\begin{proposition} \label{prop:arithmetic_surfaces}Let $L$ be a nef line bundle on a normal integral scheme $X$ admitting a projective morphism $\pi \colon X \to S$. Assume that the (absolute) dimension of $X$ is two, {$\dim \pi(X) \geq 1$ and $L|_{X_{\eta}}$ is semiample, where $\eta$ is the generic point of $\pi(X)$}. Then $L$ is EWM, and if {positive characteristic closed points of $S$ are locally finite fields,} then $L$ semiample.
\end{proposition}
\begin{proof}
By Stein factorisation, we can assume that {$\pi$ is surjective and} $\pi_* \mcO_X = \mcO_{S}$, where $\pi \colon X \to S$ is the projection. In particular, we may assume that $S$ is integral and normal. We divide the proof into two cases depending on whether {$L$ is big} or not.

In the former case, we can apply Theorem \ref{theorem:main} and reduce to showing that  $L|_{\mathbb{E}(L)}$ is EWM (semiample, resp.). But $\mathbb{E}(L)$ is a scheme of dimension at most one, and so $L|_{\mathbb{E}(L)}$ is EWM (semiample if {positive characteristic closed points of $S$ are locally finite fields). Here, we used that $L|_{X_{\mbQ}}$ is semiample as either $X_{\mbQ} = \emptyset$, or  $L|_{X_{\mbQ}}$ is big and $X_{\mbQ}$ is an integral normal scheme of dimension at most one.}

In the latter case, $\dim S=1$, {and since $S$ is normal, it is regular by \cite[Tag 0BX2]{stacks-project}}.
Moreover, {$L|_{X_{\eta}} \sim_{\mbQ} 0$}, so we can apply \cite[Lemma 2.17]{ct17} to deduce that in fact $L$ is relatively torsion. 
\end{proof}

The proof of Corollary \ref{cor:bpf} follows  exactly the same strategy as in \cite{keel99}. For the convenience of the reader, we attach a sketch of the proof below, following a slight reformulation of it as written in \cite{MNW15}.
\begin{corollary}[{Corollary \ref{cor:bpf_intro}}] \label{cor:bpf} Let $(X,\Delta)$ be a klt pair on a normal integral scheme $X$ of (absolute) dimension three which is projective {and surjective} over a spectrum $S$ of a mixed characteristic Dedekind domain with perfect residue fields of closed points. Let $L$ be a nef {and big} Cartier divisor on $X$ such that $L-(K_X+\Delta)$ is nef and big. Then $L$ is EWM. If the residue fields of { positive characteristic closed points of $S$ are locally finite}, then $L$ is semiample.
\end{corollary}
\begin{proof}
{We prove that $L$ is semiample (resp.\ EWM). In the semiample case, we assume that positive characteristic closed points of $S$ (and so also of $X$ as it is projective over $S$) are locally finite.} By taking a Stein factorisation, we can assume that $\pi_*\mcO_X=\mcO_S$, where $\pi \colon X \to S$ is the projection. By the base point free theorem in characteristic zero, we have that $L|_{X_{\mbQ}}$ is semiample. Since $L$ is a big Cartier divisor, up to multiplying $L$ by some number, we can decompose it as $L \sim_{\mathbb{Q}} A + D$, 
where $A$ is an ample and $D$ is an effective Cartier divisor. By Theorem \ref{theorem:main},
it is enough to show that $L|_{D_{\red}}$ is semiample, where $D_{\red}$ is the reduction of $D$.

Write $D_{\red} = \sum_{i=1}^m D_i$, 
where $D_i$ are prime divisors and define $\lambda_i \in \mathbb{Q}$ so that $\Delta + \lambda_i D$ contains $D_i$ with coefficient one. In particular, there exists 
an effective $\mathbb{Q}$-divisor $\Gamma_i$ such that 
\[
\Delta + \lambda_i D = D_i + \Gamma_i
\]
and $D_i \not \subset \Supp(\Gamma_i)$. 
Since $(X,\Delta)$ is klt, 
it follows that $\lambda_i > 0$.
By rearranging indices, we may assume without loss of generality that $\lambda_1 \leq \lambda_2 \leq \ldots \leq \lambda_m$, so we have 
\[
\sum_{1 \le j \le i-1} D_j \leq \Gamma_i 
\]
for each $i$.  We define $U_0 \coloneq \emptyset$ and $U_i \coloneq U_{i-1} \cup D_i$ for $i>0$. We prove that $L|_{U_i}$ is semiample {(resp.\ EWM)} by induction on $i$. By adjunction, 
there exists an effective $\mbQ$-divisor $\Delta _{\overline{D}_i}$ such that 
$
(K_X + D_i + \Gamma _i)|_{\overline{D}_i} \sim K_{\overline{D}_i} + \Delta _{\overline{D}_i},
$
where $\overline D_i \to D_i$ is the normalisation. {Note that  $(1+\lambda_i)L|_{\overline D_i} = K_{\overline D_i}+\Delta _{\overline{D}_i}+A'$ for $A'$ ample (see e.g.\ the proof of Lemma \ref{lemma:FDelta<2}).}
 
Let us assume that $L|_{U_{i-1}}$ is semiample {(resp.\ EWM)}. We first prove that $L|_{\overline D_i}$ is semiample {(resp.\ EWM)}. If $D_i$ is of mixed characteristic, then $L|_{\overline D_i}$ is semiample (resp.\ EWM) by Proposition~\ref{prop:arithmetic_surfaces}. {If $D_i$ is of positive characteristic but $\dim \pi(D_i)\geq 1$, then $L|_{(\overline D_i)_{\eta}}=K_{(\overline{D}_i)_{\eta}} + \Delta_{(\overline{D}_i)_{\eta}} + A'|_{(\overline{D}_i)_{\eta}}$ is semiample for the generic point $\eta$ of $\pi(D_i)$ by the base point free theorem for curves (equiv.\ classification of curves), and so $L|_{\overline D_i}$ is semiample (resp.\ EWM) by Proposition~\ref{prop:arithmetic_surfaces}.} If $D_i$ is {projective} over a positive characteristic field, then $L|_{\overline D_i}$ is semiample (resp.\ EWM) by an analogous argument to that in \cite[p.\ 279]{keel99}. 
For the convenience of the reader, we summarise this argument briefly. When $L|_{\overline D_i}$ has numerical dimension zero, we are done by assumptions. When the numerical dimension is two, this follows by Theorem \ref{theorem:main}. When the numerical dimension is one, a Riemann-Roch calculation as in \cite[p.\ 280]{keel99} shows that $\chi(\mcO_{\overline D_i}(mL|_{\overline D_i}))$ grows linearly with $m$. Up to a base change, we can assume that $D_i$ is defined over an algebraically closed field. Thus, {$L|_{\overline D_i}$} is semiample {(in both cases)} by \cite[Lemma 5.2 and 5.4]{keel99} {(the latter reference gives a bound on $H^2(\overline D_i, \mcO_{\overline D_i}(mL|_{\overline D_i}))$, which then implies that some multiple of $L|_{\overline D_i}$ is linearly equivalent to an effective divisor; the former reference states that $\mbQ$-effective nef line bundles of numerical dimension one on normal projective surfaces over algebraically closed fields are semiample). }

Assume $\kappa(L|_{\overline{D}_i})$ is equal to $0$ or $2$. 
Then the assumptions of Proposition \ref{prop:gluing_residue_field_finite} are satisfied, and so $L|_{D_i}$ is semiample {(resp.\ EWM)}. 
Using the same proposition again for $X_1 = U_{i-1}$ and $X_2 = D_i$, we get that $L|_{U_i}$ is semiample {(resp.\ EWM)}. 

In what follows, we assume $\kappa(L|_{\overline{D}_i}) = 1$.
\begin{lemma}\label{lemma:FDelta<2}
Let $\pi _i \colon \overline{D}_i \to Z_i$ be 
the map associated to the semiample line bundle $L|_{\overline{D}_i}$ and 
let $F$ be {the generic} fibre of $\pi _i$. 
Further, let $C_i \subset \overline{D}_i$ be the the reduction of the conductor of the normalisation 
$p_i \colon \overline{D}_i \to D_i$.
Then $C_i|_F$ is {a geometrically connected zero-dimensional scheme}.  
\end{lemma}
\noindent {In what follows the degrees of line bundles on $F$ are taken with respect to $L=H^0(F, \mcO_F)$.} 
\begin{proof}
For $M_i \coloneq (1 + \lambda_i)L - (K_X + \Delta + \lambda_i D)$, it holds that
\begin{align*}
	M_i & =L - (K_X + \Delta) + \lambda_i (L - D) \\
		&\sim_{\mathbb{Q}} (L - (K_X + \Delta)) + \lambda_i A,
\end{align*}
and so $M_i$ is ample, because $L - (K_X + \Delta)$ is nef and $\lambda_i A$ is ample. In particular, {$\deg M_i|_{F} > 0$}. 
Since {$\deg L|_{F}=0$}, 
we get {$\deg (K_{\overline{D}_i} + \Delta_{\overline{D}_i})|_F < 0$.}
Hence, 
\[
{\deg \Delta_{\overline{D}_i}|_F < -\deg K_{\overline{D}_i}|_F = -\deg K_F=2,}
\]
{where the last equality follows from the fact that $F$ is a conic over $L$ (cf.\ \cite[Chapter 9, Proposition 3.16]{liu02}).} By the adjunction formula, 
the one-dimensional part of $C_i$ is contained in $\Supp (\lfloor \Delta_{\overline{D}_i} \rfloor)$. 
Hence, we get $\deg C_i|_F \leq {\deg \Delta_{\overline{D}_i}|_F} < 2$. 
\end{proof}

By this lemma, the assumptions of Proposition \ref{prop:gluing_residue_field_finite} are satisfied, and so $L|_{D_i}$ is semiample {(resp.\ EWM)}. 
Let $\rho _i \colon D_i \to Z_i'$ be the map associated to $L|_{D_i}$ 
and let $G$ be a {generic fibre} of $\rho _i$. We get the following commutative diagram, where $\pi_i \colon \overline{D}_i \to Z_i$ is the map associated to $L|_{\overline{D}_i}$:
\begin{center}
\begin{tikzcd}
\overline{D_i} \ar{d}{\pi _i} \ar{r}{p_i} & D_i \ar{d}{\rho _i}  \\
Z_i \ar{r}  & Z_i'.
\end{tikzcd}
\end{center}

We want to apply Proposition \ref{prop:gluing_residue_field_finite} to $X_1 = U_{i-1}$ and 
$X_2 = D_i$ to show that $L|_{U_i}$ is semiample {(resp.\ EWM)}. 
It is sufficient to prove that $G$ intersects $U_{i-1}\cap D_i$ in at most one point. 

By definition of $U_i$ and the adjunction formula, the one-dimensional part $\Lambda$ of 
$p_i^{-1}(U_{i-1} \cap D_i)$ is contained in $\Supp(\lfloor \Delta_{\overline{D}_i} \rfloor)$.  By the proof of Lemma \ref{lemma:FDelta<2}, we can conclude 
{\[
 \deg \Lambda|_G  
\leq  \deg \Delta_{\overline{D}_i}|_F  < 2, 
\]}
which completes the proof. 
\end{proof}

\section*{Acknowledgements}
We are indebted to Piotr Achinger for numerous helpful conversations, to Paolo Cascini for pointing out the question of constructing mixed characteristic contractions and for his constant support throughout the project, to J{\'a}nos Koll{\'a}r for many fruitful discussions and his help in constructing Example \ref{example:extending_fails_in_char_0}, and to Maciek Zdanowicz for important conversations at the onset of the project which allowed us to kickstart it. We are also very grateful to Bhargav Bhatt, Elden Elmanto, Christopher Hacon, Giovanni Inchiostro, Kiran Kedlaya, Linquan Ma, Akhil Mathew, Zsolt Patakfalvi, Karl Schwede, Hiromu Tanaka, Joe Waldron, and Shou Yoshikawa for comments and helpful suggestions. We thank the referees for their very valuable comments and for reading the article thoroughly.

The author was supported by the Engineering and Physical Sciences Research Council [EP/L015234/1] during his PhD at Imperial College London, by the National Science Foundation under Grant No.\ DMS-1638352 at the Institute for Advanced Study in Princeton, and by the National Science Foundation under Grant No.\ DMS-1440140 while the author was in residence at the Mathematical Sciences Research Institute in Berkeley, California, during the Spring 2019 semester.
The author is currently supported by the NSF research grant no: DMS-2101897.
\bibliography{final}

\renewcommand{\refname}{\rule{2cm}{0.4pt}}

\end{document}